\documentclass[12pt,twoside,a4paper]{amsart}
\usepackage{amssymb}
\date{\today}

\def\w{\wedge}

\def\dbar{\bar\partial}

\def\R{{\mathbb R}}
\def\C{{\mathbb C}}
\def\w{{\wedge}}

\def\B{{\mathbb B}}

\def\S{{\mathcal S}}

\def\F{{\mathcal F}}
\def\Pr{{\mathcal P}}
\def\K{{\mathcal K}}

\def\Hom{{\rm Hom\, }}
\def\codim{{\rm codim\,}}

\def\Im{{\rm Im\, }}

\def\K{{\mathcal K}}
\def\Ker{{\rm Ker\,  }}
\def\rank{{\rm rank\, }}

\def\E{{\mathcal E}}

\def\Ok{{\mathcal O}}

\def\L{{\mathcal L}}

\def\Re{{\rm Re\,  }}

\def\L{{\mathcal L}}

\def\J{{\mathcal J}}
\def\nbh{neighborhood }
\def\PM{{\mathcal{PM}}}
\def\HM{{\mathcal{PM}}}
\def\be{\begin{equation}}
\def\ee{\end{equation}}

\newtheorem{thm}{Theorem}[section]
\newtheorem{lma}[thm]{Lemma}
\newtheorem{cor}[thm]{Corollary}
\newtheorem{prop}[thm]{Proposition}

\theoremstyle{definition}

\theoremstyle{remark}

\newtheorem{preremark}{Remark}
\newtheorem{preex}{Example}

\newenvironment{remark}{\begin{preremark}}{\qed\end{preremark}}
\newenvironment{ex}{\begin{preex}}{\qed\end{preex}}

\numberwithin{equation}{section}

\title[]{Koppelman formulas and 
the $\dbar$-equation  on  an analytic space}

\begin{document}

\date{\today}

\author{Mats Andersson \& H\aa kan Samuelsson}

\address{Department of Mathematics\\Chalmers University of Technology and the University of 
Gothenburg\\S-412 96 G\"OTEBORG\\SWEDEN}

\email{matsa@chalmers.se, hasam@chalmers.se}


\thanks{The first author was
  partially supported by the Swedish 
  Research Council; the second author was partially supported by
  a Post Doctoral Fellowship from the Swedish 
  Research Council. Both authors thanks the Institut Mittag-Leffler 
  (Djursholm, Sweden) where part of this work was done.}

\begin{abstract}
Let $X$ be an analytic space of pure dimension.
%%subvariety of pure codimension
%%$p$ of a pseudoconvex set $\Omega$ in $\C^n$. 
We introduce  a formalism  to generate intrinsic weighted Koppelman formulas
on $X$ that provide  solutions to the $\dbar$-equation. 
We prove that if $\phi$ is a smooth  
$(0,q+1)$-form on a Stein space $X$ with $\dbar\phi=0$, then there is 
a smooth $(0,q)$-form $\psi$ on $X_{reg}$ with at most polynomial
growth at $X_{sing}$ such that $\dbar\psi=\phi$.
The integral formulas also give  other new existence results
for the  $\dbar$-equation and Hartogs theorems, 
as well as new proofs of various known results.
\end{abstract}

\maketitle

\section{Introduction}

Let $X$ be an analytic space  of pure dimension $d$
%%of a domain  $\Omega$   in $\C^n$ 
and let $\Ok_X$ be the structure sheaf of (strongly)
holomorphic functions.
Locally  $X$ is a subvariety of
a domain  $\Omega$ in $\C^n$ and then 
$\Ok_X=\Ok/\J$, where $\J$ is the sheaf in $\Omega$ 
of holomorphic functions that vanish on  $X$.
In the same way we say that $\phi$ is  a  smooth $(0,q)$-form on $X$,
$\phi\in\E_{0,q}(X)$, if given a local embedding,  there is 
a  smooth form in a \nbh in the ambient space 
such that $\phi$ is its pull-back to $X_{reg}$.
It is well-known that this defines an intrinsic sheaf
$\E_{0,q}^X$ on $X$. 
It was proved in \cite{HPo} that if 
$X$ is embedded as a reduced complete intersection (see Example~\ref{ulla})
in a pseudoconvex domain
and $\phi$ is a $\dbar$-closed smooth form on  $X$, then 
there is a solution $\psi$ to $\dbar\psi=\phi$ on $X_{reg}$.
It has been an open question since then whether this
holds  more generally.  In this paper we prove that 
this   is indeed true for  any Stein  space $X$.

We introduce Koppelman formulas with weight factors
on $X$ by means of which we can obtain intrinsic solutions 
operators for the $\dbar$-equation. 
We begin with  a semi-global existence result. 

\begin{thm}\label{main} Let $Z$ be an analytic subvariety of pure dimension of
a pseudoconvex domain $\Omega\subset\C^n$ and assume that $\omega\subset\subset \Omega$.
There are linear operators $\K\colon\E_{0,q+1}(Z)\to\E_{0,q}(Z_{reg}\cap\omega)$
and $\Pr\colon \E_{0,0}(Z)\to\Ok(\omega)$ such that
\begin{equation}\label{koppelman}
\phi(z)=\dbar \K\phi(z)+\K(\dbar\phi)(z),\quad z\in Z_{reg}\cap\omega,
\phi\in\E_{0,q}(Z),\ q>0,
\end{equation}
and
\begin{equation}\label{koppelman2}
\phi(z)= \K(\dbar\phi)(z)+ \Pr\phi(z),\quad z\in Z_{reg}\cap\omega,\ \phi\in\E_{0,0}(Z).
\end{equation}
Moreover, there is a number $M$ such that 
\begin{equation}\label{asymp}
\K\phi(z)=\Ok(\delta(z)^{-M}),
\end{equation}
where $\delta(z)$ is the distance to $Z_{sing}$.
\end{thm}

The operators are given as
\begin{equation}\label{kpv}
\K\phi(z)=\int_\zeta K(\zeta,z)\w \phi(\zeta),
\quad
\Pr\phi(z)=\int_\zeta P(\zeta,z)\w \phi(\zeta),
\end{equation}
where $K$ and $P$ are intrinsic integral kernels  
on $Z\times (Z_{reg}\cap\omega)$ and $Z\times \omega$,
respectively. They are locally integrable with respect
to $\zeta$ on $Z_{reg}$ and the integrals in \eqref{kpv} 
are principal values at $Z_{sing}$.
%%The equality \eqref{koppelman} holds  in the sense of intrinsic
%%forms on the manifold $Z_{reg}\cap\omega$.
If $\phi$ vanishes in a neighborhood of a point $x$, then
$\K\phi$ is smooth at $x$. 
%%The  distance is the one induced from the ambient space;
%%up to  constants it is independent of the particular embedding.

There is an integer $N$ only depending on
$Z$ such that $\K\colon C^k_{0,q+1}(Z)\to C^k_{0,q}(Z_{reg}\cap \omega)$
for each $k\ge N$
and
$\Pr\colon C^k_{0,0}(Z)\to\Ok(\omega)$. Here $\phi\in C^k_{0,q}(Z)$ means that $\phi$ is the 
pullback to $Z_{reg}$ of a $(0,q)$-form of class $C^k$ in a \nbh of $Z$ in the ambient space.
As a corollary we have

\begin{cor}\label{maincor}
(i)  \   If $\phi\in C^k_{0,q}(Z)$, $k\ge N+1$,  and $\dbar\phi=0$, then there is 
$\psi\in C^k_{0,q}(Z_{reg}\cap\omega)$ with
$\psi(z)=\Ok(\delta(z)^{-M})$ and $\dbar\psi=\phi$.

\smallskip

\noindent (ii)\  If $\phi\in\ C^{N+1}_{0,0}(Z)$ and $\dbar\phi=0$ then
$\phi$ is strongly holomorphic.
\end{cor}

Part (ii)  is well-known, \cite{Ma} and \cite{Sp}, but 
$\Pr\phi$ provides an explicit holomorphic extension of $\phi$ to $\omega$.
The existence result in \cite{HPo} for a reduced complete intersection
is also obtained by an integral formula, which however 
does not give an intrinsic solution operator on $Z$.

\smallskip

We cannot expect our solution $\K\phi$ to be smooth across $Z_{sing}$.
For instance,
let $Z$ be the  germ of a curve at $0\in\C^2$ defined by $t\mapsto (t^3, t^7+t^8)$.
If   $\phi=\bar w d\bar z=3(\bar t^9+\bar t^{10})d\bar t$
then there is no solution $\psi=f(t^3,t^7+t^8)$ with $f$ smooth.
See \cite{RuppDiss} for other examples.
However, it turns out that the difference of two of our solutions is
anyway $\dbar$-exact on $Z_{reg}$ if $q>1$ and strongly holomorphic if $q=1$. 
By  an elaboration of these facts we can prove:

\begin{thm}\label{mainglobal}
Assume that $X$ is an analytic space of pure dimension. 
Any smooth $\dbar$-closed $(0,q)$-form $\phi$  on
$X$, $q\ge 1$,  defines a canonical class in $H^q(X,\Ok_X)$, and if this class
vanishes then there is a global smooth form $\psi$ on $X_{reg}$
such that $\dbar\psi=\phi$. 
In particular, there is always such a solution if
$X$ is a  Stein space.
\end{thm}

We can use our integral formulas to  
solve the  $\dbar$-equation  with compact support. 
As usual this leads to 
Hartogs results for holomorphic functions.

\begin{thm}\label{portensats}
Assume  that $X$ is a  Stein space of pure dimension $d$
with globally irreducible components $X^\ell$ 
and let $K$ be  compact subset such that $X_{reg}^\ell\setminus K$ is connected
for each $\ell$. Let   $\nu$ be the  (minimal) 
depth of the rings $\Ok_{X,x}$,  $x\in X_{sing}$.

\smallskip
\noindent (i)\quad If $\nu\ge 2$, then for each holomorphic
function $\phi\in\Ok(X\setminus K)$ there is
$\Phi\in\Ok(X)$ such that $\Phi=\phi$ in $X\setminus K$.

\smallskip
\noindent(ii)\quad Assume that  $\nu=1$ and let  $\chi$ be  a cutoff function
that is identically $1$ in a \nbh of $K$. 
There is a smooth $(d,d-1)$-form $\alpha$ on $X_{reg}$
such that the function
$\phi\in\Ok(X\setminus K)$ has a holomorphic extension $\Phi$
across $K$ if and only if
\begin{equation}\label{moment}
\int_Z\dbar\chi\w \alpha \phi h=0, \quad h\in\Ok(X),
\end{equation}
where the integrals exist as principal values at $X_{sing}$.
\end{thm}

If  $X$ is normal and $X\setminus K$ is connected, then the conditions
in  (i) are fulfilled, and so we get a Hartogs theorem 
that was  proved by other methods by Merker and Porten
in \cite{MePo}. Recently, Ruppenthal, \cite{RuppHartog}, also gave a proof by $\dbar$-methods
in case $X_{sing}$ discrete. If $X$ is not normal it is necessary to assume that
$X^{\ell}_{reg}\setminus K$ is connected; see Example \ref{hartogsex} in Section \ref{compsupp} below.

In the same way we can obtain the existence of $\dbar$-closed
extensions across $X_{sing}$ of $\dbar$-closed
forms in $X_{reg}$. This leads to  existence results
for the $\dbar$-equation in $X_{reg}$ via Theorem~\ref{main}.  
In this way we obtain
the following vanishing theorem  that was 
proved already in \cite{Scheja1}  by analyzing the Cech cohomology 
of the sheaf $\Ok/\J$ in a local  embedding of $X$.

\begin{thm}\label{main3}
Assume that $X$ is a Stein space of pure dimension $d$.
%%%
Let   $\nu$ be the  (minimal) 
depth of the rings $\Ok_{X,x}$,  $x\in X_{sing}$.
Assume that $\phi$ is a smooth $\dbar$-closed
$(0,q)$-form in $X_{reg}$. 
If $0<q<\nu-1-\dim X_{sing}$,  
then  there is a smooth solution to $\dbar\psi=\phi$
in $X_{reg}$.
If  $q=0<\nu -1-\dim X_{sing}$, then  $\phi$ extends  to a strongly holomorphic
function.
\end{thm}

If  $q=\nu-1-\dim X_{sing}$, then  the same conclusion is 
true if and only if a certain moment condition,
similar to \eqref{moment},
is fulfilled locally at $Z_{sing}$.  
The sufficient condition in case $q=0$ is not necessary.
The precise condition is Serre's criterion; see  Section~\ref{merosec}, where
we also present  a conjecture about an analogous sharp(er)
criterion for solvability of $\dbar$ for $q>0$.

We  have the following new vanishing result:

\begin{thm}\label{main3x}
Assume that $X$ is a Stein space of pure dimension $d$.
%%of a pseudoconvex  domain  $X\subset\C^n$.
If  $\dim X_{sing}=0$, then for each smooth $(0,d)$-form
on $X_{reg}$ there is a  smooth solution to $\dbar\psi=\phi$ on $X_{reg}$. 
\end{thm}

%%The variety $Z$, or equivalently $\Ok/\J$,  is  Cohen-Macaulay if and only if
%%$\nu=n-p=\dim Z$. In this case, thus $\nu-\dim Z_{sing}$ is the codimension
%%of $Z_{sing}$ in $Z$. In particular, if $Z_{sing}$ is discrete there is 
%%an obstruction only when $q=\dim Z-1$.

If $\nu=\dim X$ (i.e.,  $X$ is Cohen-Macaulay)  and  $X_{sing}$
is discrete, then  there is  thus a local  obstruction only when 
$q=\dim X-1$  (as at a  regular point).

Our solution operator $\K$ behaves like a classical solution operator
on $X_{reg}$ and by  appropriate weights
we get

\begin{thm}\label{main2}
Assume that $Z$ is subvariety of pure dimension  of a pseudoconvex domain
$\Omega\subset\C^n$ and let $\omega\subset\subset \Omega$.
Given $M\ge 0$ there is an $N\ge 0$ and a linear operator 
$\K$ such that if $\phi$ is a $\dbar$-closed $(0,q)$-form on $Z_{reg}$ with
$\delta^{-N}\phi \in L^p(Z_{reg})$, $1\leq p\leq \infty$, then
$\dbar\K\phi=\phi$ and $\delta^{-M}\K \phi \in L^p(Z_{reg})$.
\end{thm}

The existence of such solutions was  proved  in \cite{FOV2} 
(even for $(r,q)$-forms) by resolution of singularities
and cohomological methods (for $p=2$, but the same method surely gives
the more general results).
%%As in \cite{FOV2} one can replace $Z_{sing}$ by any analytic subvariety $A$ of 
%%$Z$ that contains
%%$Z_{sing}$. 
By a standard technique this theorem implies global
results for a Stein space $X$. %%subvariety $Z$ of  a pseudoconvex  domain. 

In case $Z_{sing}$ is a single point  more precise result
are  obtained in \cite{PS} and \cite{FOV1}. In particular, if $\phi$ has bidegree
$(0,q)$, $q<\dim Z$, then the image of $L^2(Z_{reg})$ under $\dbar$ 
has finite codimension in $L^2(Z_{reg})$.
%%These  paper also contains results for $(p,q)$-forms.
%%For the precise result for $(p,q)$-forms, see \cite
See also \cite{OV}, and the references given there,
for related results.
%%There has  also been work done on getting solutions with H\"{o}lder estimates.
In \cite{FG}, Forn\ae ss and Gavosto show that, for complex curves,
a H\"{o}lder continuous solution exists if the right hand side  is bounded. Recently,
certain hypersurfaces have also been considered, e.g., in \cite{RuppDiss}.
%%It might happen that our solution formulas could be useful for such 
%%questions, but we do not pursue this theme in this paper.
%%%

In \cite{Ts} Tsikh obtained a residue criterion for a weakly
holomorphic function (or even a meromorphic function)
to be strongly holomorphic in case $Z$ is a (reduced)
complete intersection. This result was recently extended to
a general variety in \cite{A13}. By formula 
\eqref{koppelman2} we get a new proof of this result and
an explicit representation of the holomorphic extension.

The main ingredients in the construction of the integral
operators $K$ and $P$ in Theorem~\ref{main} are a 
 certain residue current $R$, introduced in  \cite{AW1} and \cite{AW2},
that is   associated to the variety $Z$,
and the  integral representation formulas from \cite{A7}.
We discuss the current $R$  
in Section~\ref{res}, 
%%where the principal result
%%is Proposition~\ref{regular}. 
and in  Section~\ref{koppsec} we obtain the Koppelman formula
as the restriction to $Z$ of a certain global formula in 
the ambient set $\Omega$.
%%In Section~\ref{pfs} we compare different local solution
%%obtained in this way and prove Theorems~\ref{main}
%%and \ref{mainglobal}.
%%%%%%
In Section \ref{exsec} we compute our  Koppelman
formulas more explicitly in case $Z$ is a reduced complete
intersection.
%%mean in the case of a variety defined by a complete intersection.
%%This is done in some detail in order to facilitate the reading of the
%%previous sections. 
The resulting formula for $\Pr$ coincides with the representation
formula by Stout \cite{Stout} and Hatziafratis \cite{Hatzia} when
$Z_{sing}$ is discrete.
%%We also do explicit computations for the cusps
%%$\{z_1^r=z_2^s\}\subset \C^2$ with $r$ and $s$ relatively prime.
%%The resulting formulas for the operators
%%$K$ and $P$ are surprisingly simple.

%%In case $Z$ is an algebraic variety in $\C^n$  we get a global version
%%of Theorem~\ref{main}; this is described in 
%%Section~\ref{algsec}.

\smallskip
{\bf Acknowledgement:} We are indebted to Jean Ruppenthal and Nils
\O vrelid for important remarks on an earlier version of this paper.

\section{A residue current associated to $Z$}\label{res}

Let $Z$ be a subvariety  of pure codimension $p=n-d$ 
of a pseudoconvex set $\Omega\subset\C^n$. 
The Lelong current $[Z]$ is a  classical 
analytic object that represents $Z$.
It is a $d$-closed $(p,p)$-current such that 
$$
[Z].\xi =\int_Z\xi
$$
for test forms $\xi$. If $\codim Z=1$,  $Z=\{f=0\}$ and  $df\neq 0$ on $Z_{reg}$, then
a simple form  of the Poincare-Lelong formula states that 
\begin{equation}\label{pl}
\dbar\frac{1}{f}\w \frac{df}{2\pi i}=[Z].
\end{equation}
To construct integral formulas  we will use an analogue of the current $\dbar(1/f)$, 
introduced in \cite{AW1}, for a
general variety $Z$. It turns out that this current,
contrary to $[Z]$,   also reflects certain  subtleties of the variety
at $Z_{sing}$ that are encoded by the  algebraic description of $Z$.
Let  $\J$ be the ideal sheaf over $\Omega$ generated by the variety $Z$.
In a slightly smaller set, still denoted $\Omega$, one can find a free resolution 
\begin{equation}\label{acomplex}
0\to \Ok(E_N)\stackrel{f_N}{\longrightarrow}\ldots
\stackrel{f_3}{\longrightarrow} \Ok(E_2)\stackrel{f_2}{\longrightarrow}
\Ok(E_1)\stackrel{f_1}{\longrightarrow}\Ok(E_0)
\end{equation}
of the sheaf $\Ok/\J$.
%%$J=\Im(\Ok(E_1)\to\Ok(E_0)$ is a subsheaf of $\Ok(E_0)$. 
Here $E_k$ are trivial vector bundles over $\Omega$
and $E_0=\C$ is the trivial line bundle.
This resolution induces a complex  of trivial vector bundles
\begin{equation}\label{bcomplex}
0\to E_N\stackrel{f_N}{\longrightarrow}\ldots
\stackrel{f_3}{\longrightarrow} E_2\stackrel{f_2}{\longrightarrow}
E_1\stackrel{f_1}{\longrightarrow}E_0\to 0
\end{equation}
that is pointwise exact outside $Z$.
%%By Hilbert's syzygy theorem we may assume that $N\le n$. 
Let $Z_k$ be the set where $f_k$ does not have optimal rank. Then
$$
\cdots Z_{k+1}\subset Z_k\subset\cdots \subset Z_p=Z,
$$
and these sets are independent of the choice of resolutions, thus invariants
of the sheaf $\F=\Ok/\J$. The Buchsbaum-Eisenbud theorem claims
that $\codim Z_k\ge k$ for all $k$, and since 
furthermore  $\F$ has pure codimension $p$ in our case, 
$Z_k\subset Z_{sing}$ for $k>p$,
and (see Corollary~20.14 in \cite{Eis1})
\begin{equation}\label{olvon}
\codim Z_k\ge k+1, \quad  k\ge p+1.
\end{equation}
%%The sets $Z_k$ describe in a sense
%%the complexity of the singularites of $Z$.
There is a resolution \eqref{acomplex} 
if and only if $Z_k=\emptyset$ for $k>N$, and this number
is equal to $n-\nu$, where $\nu$ is the minimal depth
of $\Ok/\J$.
In particular, the variety is Cohen-Macaulay, or equivalently,
the sheaf $\F=\Ok/\J$ is Cohen-Macaulay if and only if
$Z_k=\emptyset$ for $k\ge p+1$. In this case we can
thus choose the resolution so that $N=p$.

\begin{remark}\label{obsz}
Let us define  $Z^0=Z_{sing}$ and $Z^r=Z_{p+r}$
for $r>0$. One can prove that  these sets are independent of the embedding
and thus intrinsic objects of the analytic space $Z$
that describe the complexity of the singularities. 
In fact, by the uniqueness
of minimal embeddings, it is enough to verify that these  sets are unaffected
if we add nonsense variables and consider $Z$ as embedded into
$\Omega\times\C^m$. This follows,  e.g., from the proof
of  Theorem~1.6 in \cite{A13}.
\end{remark}

Given Hermitian metrics on  $E_k$  in \eqref{acomplex} in \cite{AW1} 
was defined  a 
current $U=U_1+\cdots+U_n$, where $U_k$ is a $(0,k-1)$-current with values in
$E_k$,  and a residue current with support on $Z$,
\begin{equation}\label{rd}
R=R_p+R_{p+1}+\cdots+R_N,
\end{equation}
where $R_k$ is a $(0,k)$-current with values in $E_k$,
satisfying
$$
\nabla_f U=1-R,
$$
if  $\nabla_f=f-\dbar= \sum f_j-\dbar$.
Outside $Z$, the current  $U$ is a smooth form $u$ and
if $F=f_1$, then   
$U=|F|^{2\lambda}u|_{\lambda=0}$ and $R=\dbar|F|^{2\lambda}\w u|_{\lambda=0}$.
In case $Z$ is Cohen-Macaulay and $N=p$, then
$R=R_p$ is $\dbar$-closed.

\begin{ex}\label{ulla}
Assume that  $Z$ is a reduced complete intersection, i.e.,
defined by $a=(a_1,\ldots,a_p)$ with $da_1\w\ldots \w da_p\neq 0$ on $Z_{reg}$. 
Then  the Koszul complex induced by $a$ provides
a resolution of $\Ok/\J$. Let $e_1,\ldots,e_p$ be a  holomorphic frame
for the trivial bundle $A$ and consider $a$ as the section $a=a_1e_1^*+a_2e_2^*+\cdots$
of the dual bundle $A^*$, where $e_j^*$ is the dual frame.
Let $E_k=\Lambda^k A$, and let all the mappings $f_k$ in \eqref{acomplex}
be interior multiplication, $\delta_a$, with $a$.
Notice that  $s_a=\sum_j \bar a_j e_j/|a|^2$ is the minimal solution to 
$\delta_a s_a=1$ outside $Z$ (with respect to the trivial metric on $A$). 
If we consider all forms as sections of the bundle $\Lambda(T^*(\Omega)\oplus A)$,
see \cite{AW1}, then
$u_k=s_a\w(\dbar s_a)^{k-1}$.
If $F$ is any holomorphic tuple such that $|F|\sim |a|$, then, see, e.g., \cite{AW1},
\begin{equation}\label{pdum}
R=R_p=\dbar|F|^{2\lambda}\w u_p\big|_{\lambda=0}=\dbar\frac{1}{a_p}\w\ldots\w\dbar\frac{1}{a_1}
\w e_1\w\ldots\w e_p,
\end{equation}
i.e., the classical Coleff-Herrera product (times $e_1\w\ldots\w e_p$).
It is wellknown that 
\begin{equation}\label{baka}
\dbar\frac{1}{a_p}\w\ldots\w\dbar\frac{1}{a_1}\w da_1\w\ldots\w da_p/(2\pi i)^p=[Z].
\end{equation}
For further reference we also observe  that 
\begin{equation}\label{anita}
\dbar|F|^{2\lambda}\w u_p\to \dbar\frac{1}{a_p}\w\ldots\w\dbar\frac{1}{a_1}
\end{equation}
as measures in $Z_{reg}$ when $\lambda\searrow 0$.
This is easily verified since  we may assume that $a$ is part of a holomorphic coordinate
system.
%%After an  appropriate resolution of singularities we may assume that 
%%$a=a_0a'$, where $a_0$ is a holomorphic monomial  and $a'$ is a holomorphic tuple
%%such that $a'\neq 0$. Moreover we may assume as well that $F=F_0 F'$ in the same
%%way, and it then follows that $F_0=ca_0$ where $c\neq 0$, so we can assume that $c=1$.
%%Now $u^a_p=\omega/a_0^p$ and 
%%$$
%%da=a_0^{p-1}da_0\w\alpha +a_0^p\beta,
%%$$
%%where $\omega,\alpha,\beta$ are smooth. From the corresponding one-variable
%%statement we find that 
%%$$
%%\dbar(|a_0|^{2\lambda}|F'|^{2\lambda})\w\frac{\omega\w da_0\w\alpha+a_0\beta}{a_0}
%%\to [a_0=0]\w\omega\w\alpha
%%$$
%%as measures when $\lambda\searrow 0$, and also that the left hand side
%%has an analytic continuation whose value at $\lambda=0$ is equal to
%%the right hand side. Thus \eqref{anita} is proved in the resolution,
%%and since convergence as  measures is preserved under push-forwards,
%%\eqref{anita} holds in the original set. 
%%
\end{ex}

In \cite{AW2} was  introduced the sheaf of {\it pseudomeromorphic}  currents
$\PM$ and it was pointed out that the currents  $U$ and $R$ are pseudomeromorphic. 
%%The sheaf $\PM$ is closed under $\dbar$ and multiplication with smooth forms.
%%%
For each pseudomeromorphic current
$\mu$ and any subvariety $V$ there is a natural 
restriction $\mu{\bf 1}_V$ to $V$. 
If $h$ is a holomorphic tuple such that $V=\{h=0\}$, then
$|h|^{2\lambda}\mu$, a~priori defined when $\Re\lambda>>0$,  has a
current-valued analytic continuation to $\Re\lambda>-\epsilon$, and
the value at $\lambda=0$ is precisely $\mu- \mu{\bf 1}_{V}$.
The current $\mu{\bf 1}_V$ is again in $\PM$ and it has support on
$V$.
%%
%%Zariski-closed or open (or  constructible)  
%%set $A$, there is a
%%hypermeromorphic current $\mu {\bf 1}_A$, called  the {\it restriction} of $\mu$
%%to $A$, that has support on the closure of $\supp\mu\cap A$, and
%%$
%%\mu=\mu{\bf 1}_A+\mu{\bf 1}_{X\setminus A}.
%%$
%%Moreover, the operation $\mu\mapsto \mu{\bf 1}_A$ commutes
%%with multiplication by  smooth forms. 
%%If $V$ is closed, i.e., an analytic variety, in $\Omega$ then
%%$\mu=\mu{\bf 1}_V$, i.e., $\mu{\bf 1}_{\Omega\setminus V}=0$,
%%if and only if the natural restriction $\mu|_{\Omega\setminus V}$
%%of $\mu$ to the open set $\Omega\setminus V$ is zero.
The following property  is crucial.

\begin{prop}\label{hyp}
If $\mu\in\HM$ with  bidegree $(r,p)$ has  support on
a variety $V$ of codimension $k>p$ then $\mu=0$. 
\end{prop}

It is proved in \cite{AW2} that the restriction $R{\bf 1}_V$ of $R$ to any subvariety 
$V$ of $Z$ (of higher codimension) must vanish; we say that
$R$ has the {\it standard extension property}, SEP, with respect to $Z$.
For the component $R_p$ of $R$ the SEP  follows immediately
from Proposition~\ref{hyp}, but the general statement is deeper;
it depends on the assumption that $Z$ has pure codimension. 
In particular, if $h$ is a holomorphic function that does not vanish identically
on any component of $Z$ (the interesting case is when  $\{h=0\}$ contains
$Z_{sing}$), 
and $\chi$ is a smooth approximand of the characteristic
function for $[1,\infty)$, then
\begin{equation}\label{sepp}
\lim_{\delta\to 0}\chi(|h|/\delta) R=R.
\end{equation}

\begin{prop}\label{regular}
For the residue current  $R$  associated to \eqref{acomplex}
the following hold:

\smallskip
\noindent (i)  \  There  are smooth currents $\gamma_k$
on $Z_{reg}$ such that 
\begin{equation}\label{gammak}
R_k=\gamma_k\lrcorner [Z]
\end{equation}
there. Moreover, there is a number $M>0$ such that
\begin{equation}\label{asympupp}
|\gamma_k|\le C\delta^{-M},
\end{equation}
where $\delta$ is the distance to $Z_{sing}$.

%%\smallskip
%%\noindent (ii)  \
%%The current $R$ has the SEP; in particular $R{\bf 1}_{Z_{sing}}=0$.

\smallskip
\noindent (ii) \  If $\Phi$ is a smooth $(0,q)$-form  whose
pull-back to $Z_{reg}$ vanishes, then $R\w \Phi=0$.
\end{prop}

To be precise, $\gamma_k$ is a section of the bundle
$\Lambda^{0,k-p}T^*(X)\otimes E_k\otimes\Lambda^{p}T_{1,0}(X)$.
Part (ii) means that for each $\phi\in\E_{0,q}(Z)$  we have an intrinsically
defined current $R\w\phi$.

\begin{proof}
In a \nbh of a given point $x\in Z_{reg}$
we can choose coordinates $(w',w'')$ such that
$Z=\{w_1''=\ldots=w_p''=0\}$. 
%In fact, we can even select  $w_j''$ 
%%from the tuple $f_1=(f_1^j)$. 
Then $\J$ is generated by
$w_j''$,  the associated Koszul complex provides a (minimal)
resolution of $\Ok/\J$ there, and  the corresponding
residue current $R=R_p$ is just the Coleff-Herrera product
formed from the tuple $a=w''$, see Example~\ref{ulla} above.
An arbitrary resolution at $x$ will  contain the Koszul complex
as a direct summand, and it follows, see Theorem~4.4 in \cite{AW1} 
or Section~\ref{asympsec} below,  that  therefore
$$
R_p=\alpha \dbar\frac{1}{w_1''}\w\ldots\w
\dbar \frac{1}{w_p''},
$$ 
where $\alpha$ is a smooth $E_p$-valued form.  It follows that we can
take $\gamma_p$ as
$$
\tau=\alpha\otimes\frac{\partial}{\partial w_1''}\w\ldots\w
\frac{\partial}{\partial w''_p}/(2\pi i)^p.
$$
To obtain a global form, for  $x\in Z_{reg}$, let $L_x$ be the
orthogonal complement in $(T(X)_{1,0})_x$ of $(T(Z)_{1,0})_x$ (with respect to the
usual metric in the ambient space). We can then modify $\tau$ so that 
it takes values in $\Lambda^p L$ without affecting \eqref{gammak},
and $\gamma_p$ so defined is pointwise unique and hence a 
global smooth  form on $Z_{reg}$. For further reference we also notice that
the norm of $\gamma_p$ will not exceed the norm of the locally defined
form $\tau$.  The proof of the asymptotic estimate
\eqref{asympupp} for $k=p$ is postponed to Section~\ref{asympsec}.

Outside $Z_{k+1}$ there is a smooth $(0,1)$-form $\alpha_{k+1}$
(with values in $\Hom(E_k,E_{k+1})$) such that
$R_{k+1}=\alpha_{k+1} R_{k}$. Moreover, the 
denominator of $\alpha_{k+1}$   is the modulus square
of a tuple of subdeterminants of the matrix $f_k$, see \cite{AW1}, and 
hence $\alpha_k$  has polynomial growth when $\zeta\to Z_{k+1}$,
see \cite{AW1} Theorem~4.4.  It follows that we can take
\begin{equation}\label{kanin}
\gamma_k=\pm \alpha_k \cdots \alpha_{p+1}\gamma_p
\end{equation}
for $k\ge p+1$, and  \eqref{asympupp} for $k>p$ follows from the case $k=p$.

 To see (ii), assume that $\Phi$ vanishes on  $Z_{reg}$. 
Since $\Phi$ is  $(0,q)$ we  have that 
$
R_k\w\Phi=\gamma_k\lrcorner[Z]\w\Phi=\gamma_k\lrcorner([Z]\w\Phi)=0
$
on $Z_{reg}$. Now (ii) follows from  \eqref{sepp}.
\end{proof}

\section{Construction of  Koppelman formulas on $Z$}\label{koppsec}

We now recall the construction of integral formulas in \cite{A7}
on an open set $\Omega$ in $\C^n$. 
Let $(\eta_1,\ldots,\eta_n)$ be a holomorphic tuple in $\Omega_\zeta\times\Omega_z$
that span the ideal associated to the diagonal 
$\Delta\subset\Omega_\zeta\times\Omega_z$.
For instance, one can take $\eta=\zeta-z$. 
Following the last section in \cite{A7}  we 
consider forms in $\Omega_\zeta\times \Omega_z$ with values
in the exterior algebra $\Lambda_\eta$ 
spanned by $T^*_{0,1}(\Omega\times \Omega)$ and
the $(1,0)$-forms $d\eta_1,\ldots,d\eta_n$. 
On such forms interior multiplication $\delta_\eta$ with
$$
\eta=2\pi i \sum_1^n\eta_j\frac{\partial}{\partial \eta_j}
$$
has a meaning. We introduce  $\nabla_{\eta}=\delta_{\eta}-\dbar$.
Let $g=g_{0}+\cdots +g_{n}$   be a smooth form (in $\Lambda_\eta$)
defined for $z$ in  $\omega\subset\subset \Omega$  and $\zeta \in \Omega$, such that
$g_{0}=1$ on the diagonal $\Delta$ in $\omega\times \Omega$
(lower indices denote degree in $d\eta$) and $\nabla_\eta g=0$.
Such a form will be called a {\it weight} with respect to $\omega$.
Notice that if $g$ and $g'$ are weights, then $g\w g'$ is again  a  weight.
We  will use one weight that has compact support in
$\Omega$, and one weight which gives a division-interpolation type formula
with respect to the ideal sheaf $\J$ associated to the
variety $Z\subset\Omega$. 
%%Roughly speaking, the desired
%%Koppelman formula is then obtained by taking the restriction to
%%$Z$. 

\begin{ex}\label{alba}
If $\Omega$ is pseudoconvex and $K$ is a holomorphically convex compact
subset, then  one can find a weight with respect to some \nbh $\omega$ of $K$,
depending holomorphically on $z$,  
that has compact support (with respect to $\zeta$) 
in $\Omega$, see, e.g., Example~2 in \cite{A7}.
Here is an explicit choice  when $K$ is the closed ball $\overline{\B}$
and $\eta=\zeta-z$.
If 
$\sigma=\bar\zeta\cdot d\eta/2\pi i(|\zeta|^2-\bar\zeta\cdot z)$, 
then   $\delta_{\eta}\sigma=1$ for $\zeta\neq z$ and 
$$
\sigma\w(\dbar \sigma)^{k-1}=\frac{1}{(2\pi i)^k}
\frac{\bar\zeta\cdot d\eta\w(d\bar\zeta\cdot d\eta)^{k-1}}
{(|\zeta|^2-\bar\zeta\cdot  z)^k}.
$$
If $\chi$ is  a cutoff function that is $1$ in a slightly
larger ball, then we can take  
$$
g=\chi-\dbar\chi\w\frac{\sigma}{\nabla_{\eta} \sigma}=
\chi-\dbar\chi\w [\sigma+\sigma\w\dbar \sigma+ \sigma\w(\dbar \sigma)^2+
\cdots +\sigma\w(\dbar \sigma)^{n-1}].
$$
One can find a $g$ of the same form in the general case. %%The generaThe form
\end{ex}

Assume now that $\Omega$ is pseudoconvex.
Let us fix global frames for the bundles $E_k$ in \eqref{bcomplex} over $\Omega$.
Then $E_k\simeq\C^{\rank E_k}$,  and 
the morphisms $f_k$  are just matrices of holomorphic functions.
One can find (see \cite{A7} for explicit choices)
$(k-\ell,0)$-form-valued Hefer morphisms, i.e.,  matrices,  
$H^\ell_k\colon E_k\to E_\ell$   depending holomorphically on $z$ and $\zeta$, such that 
$H^\ell_k=0$ for $ k<\ell$, $ H^\ell_\ell=I_{E_\ell}$, and in general,
\begin{equation}\label{Hdef}
\delta_{\eta} H^\ell_{k}=
H^\ell_{k-1} f_k -f_{\ell+1}(z) H^{\ell+1}_{k};
%%%\quad k\ge\ell;
\end{equation}
here $f$ stands for $f(\zeta)$. Let
$$
HU = \sum_k H^{1}_k U^0_k,
\quad  HR=\sum_k H^0_k R_k.
$$
Thus $HU$ takes a section $\Phi$ of $E_0$, i.e., a function,
depending on $\zeta$ into a (current-valued) 
section $HU\Phi$ of $E_{1}$ depending on both $\zeta$ and $z$, and similarly,
$HR$ takes a section of $E_0$ into a section of $E_0$.
%%%We let  $HU=\sum_\ell H^\ell U$ and $H R=\sum_\ell H^\ell R$.

\smallskip

Let $s$ be a smooth $(1,0)$-form in $\Lambda_\eta$ such that $|s|\le C|\eta|$ and
$|\delta_\eta s|\ge C|\eta|^2$; such an $s$ is called  {\it admissible}. Then 
%%If  $\eta=\zeta-z$ and $b=(2\pi i)^{-1}\partial|\eta|^2/|\eta|^2$, then
$B=s/\nabla_\eta s$ is a locally integrable form and 
%%% that we refer
%%to as a  Bochner-Martinelli form, and 
\begin{equation}\label{bm}
\nabla_\eta B=1-[\Delta],
\end{equation}
where $[\Delta]$ is the $(n,n)$-current of integration over the diagonal
in $\Omega\times \Omega$. 
If $\eta =\zeta-z$, 
$s=\partial|\eta|^2$ will do, and we then  refer to the  resulting form $B$
as the Bochner-Martinelli form.

Let $g$ be any smooth weight (with respect to $\omega\subset\subset \Omega$,
but not necessarily holomorphic in $z$),
and with compact support in $\Omega$. For a smooth $(0,q)$-form $\phi$
on $Z$   we 
want to define
\begin{equation}\label{kurt5}
\K\phi(z)=\int_\zeta(HR\w g\w B)_n\w\phi, \quad z\in Z_{reg}\cap \omega,
\end{equation}
and
\begin{equation}\label{kurt8}
\Pr\phi(z)=\int_\zeta (HR\w g)_n\w \phi, \quad z\in \omega.
\end{equation}
Here the lower index denotes degree in $d\eta$. 
%%
%%\begin{prop}\label{allansbror}
To this end, let $\Phi$ be any  smooth form in $\Omega$ whose pull-back
to $Z_{reg}$ is equal to $\phi$. 
If $\Phi$ is vanishing in a \nbh of some given point
$x$ on $Z_{reg}$, then $B\w\Phi$ is smooth in $\zeta$ for $z$ close to $x$,  and 
the integral is to be  interpreted as the current
$R$ acting on a smooth form. It is clear that this  integral depends
smoothly on $z\in Z_{reg}\cap\omega$ and in view of Proposition~\ref{regular}
it only depends on $\phi$. 
Let us then assume  that  $\Phi$ has support in  a \nbh of $x$  in which
$R=\gamma\lrcorner [Z]$. Notice that 
$$
(HR\w g\w B)_n=H_p^0R_p\w(g\w B)_{n-p}+H_{p+1}^0R_{p+1}\w (g\w B)_{n-p-1}+\cdots,
$$
cf., \eqref{rd}, 
and that
\begin{equation}\label{modehatt}
(g\w B)_{n-k}=\Ok(1/|\eta|^{2n-2k-1})
\end{equation}
so it is integrable on $Z_{reg}$ for $k\ge p$.
Thus 
\begin{equation}\label{skata}
\int_\zeta H^0_kR_k\w (g\w B)_{n-k}\w\Phi=
\pm\int_{\zeta\in Z}\gamma_k\lrcorner\big(H^0_k\w (g\w B)_{n-k}\big)\w\Phi
\end{equation}
is  defined pointwise and 
depends continuously on $z\in\omega$, and it is in fact
smooth on  $Z_{reg}\cap\omega$ according to  Lemma~\ref{brum}  below.
%% that
%%it in fact depends smoothly on $z\in Z_{reg}\cap\omega$.
It is also clear from \eqref{skata} that the integral
only depends on the pullback of  $\Phi$ to $Z_{reg}$. 
In the same way one  gives a meaning to \eqref{kurt8}.

Since $B$ has bidegree $(*, *-1)$,
 $\K\phi$ is a $(0,q-1)$-form and $\Pr \phi$ is $(0,q)$-form.
It follows from \eqref{sepp} that \eqref{kpv} holds 
as principal values at $Z_{sing}$ with
\begin{equation}\label{kpv2}
K(\zeta,z)=\pm\gamma\lrcorner(H\w g\w B)_n, \quad P(\zeta,z)=\pm\gamma\lrcorner(H\w g)_n.
\end{equation}

\begin{prop}\label{koppelmanthm} Let  $g$ be  any smooth weight in $\Omega$ with 
respect to $\omega\subset\subset\Omega$ and with
compact support in $\Omega$. For any smooth $(0,q)$-form on $Z$,
$\K\phi$ is a smooth $(0,q-1)$-form in $Z_{reg}\cap\omega$, $\Pr\phi$ is a smooth
$(0,q)$-form in $\omega$, and we have the  Koppelman formula
\begin{multline}\label{koppelman3} 
\phi(z)=\\
\dbar_z\int(HR\w g\w B)_n\w\phi +
\int(HR\w g\w B)_n\w\dbar\phi +
\int(HR\w g)_n\w\phi,
\end{multline}
for $z\in Z_{reg}\cap\omega$.
\end{prop}

\begin{proof}%%[Proof of the Koppelman formula(s)]
On a formal level  the Koppelman formula follows  from
Section~7.4 in \cite{A7} by just restricting  to  $z\in Z_{reg}\cap\omega$,
but for a  strict argument one must be careful with the 
limit processes.
Let $U^\lambda=|F|^{2\lambda}u$ and 
$$
R^\lambda=\sum_{k=0}^NR^\lambda_k=1-|F|^{2\lambda}+\dbar|F|^{2\lambda}\w u,
$$
so that $\nabla_f U^\lambda=1-R^\lambda$.
We can have 
$$
g^\lambda= f(z)HU^\lambda + HR^\lambda
$$
as smooth as we want by just taking $\Re\lambda$ large enough.
If  $\Re\lambda>>0$, then, cf.,  \cite{A7} p.325,
$g^\lambda$ is a weight,  and thus,
cf., \eqref{bm}, 
$$
\nabla_\eta(g^\lambda\w g\w B)=g^\lambda\w g-[\Delta]
$$
from which we get  
\begin{equation}\label{kolja}
\dbar(g^\lambda\w g\w B)_n=[\Delta]-(g^\lambda\w g)_n.
\end{equation}
As in
\cite{A7} we  get the Koppelman formula 
\begin{equation}\label{balja}
\Phi(z)=\int_\zeta (g^\lambda\w g\w B)_n\w\dbar\Phi +
\dbar_z\int_\zeta (g^\lambda\w g\w B)_n\w\Phi +\int_\zeta 
 (g^\lambda\w g)_n\w\Phi
\end{equation}
for $z\in \omega$, and since $g^\lambda=HR^\lambda$ when $z\in Z_{reg}$ we
get
\begin{multline*}
\Phi(z)=\int_\zeta (HR^\lambda \w g\w B)_n\w\dbar\Phi +\\
\dbar_z\int_\zeta  (HR^\lambda\w g\w B)_n\w\Phi +\int_\zeta 
  (HR^\lambda\w g)_n\w\Phi,\quad z\in Z_{reg}\cap \omega.
\end{multline*}
It is now enough to check that 
\begin{equation}\label{august1}
\int_\zeta (HR^\lambda \w g\w B)_n\w\Phi, \quad \int_\zeta (HR^\lambda \w g)_n\w\Phi
\end{equation}
have   analytic continuations to $\Re\lambda >0$ and tend weakly to 
 $\K\Phi$ and $\Pr\Phi$, respectively, when $\lambda\searrow 0$. 
%%%
To this end, fix a point $x$ on $Z_{reg}\cap\omega$. If $\Phi$ vanishes identically
in a \nbh of $x$, then the first integral
in \eqref{august1} is just the current $R^\lambda$ acting
on a smooth form, and hence the continuation exists to $\Re\lambda>-\epsilon$
and has the desired value
at $\lambda=0$. Therefore, we can assume that $\Phi$ has compact support in
a \nbh of $x$ where $R=\gamma\lrcorner [Z]$. 
%%Notice that 
%%$$
%%R^\lambda= \sum_0^N R^\lambda_k=
%%1-|F|^{2\lambda} +\dbar|F|^{2\lambda} \w \sum_{k=1}^N u_k.
%%$$
Let $\psi(z)$ be a test form of bidegree $(n-p,n-p-q+1)$
with support in $Z_{reg}\cap\omega$. 
We have to prove that 
$$
\int_{z\in Z}\psi(z)\w \sum_{k=0}^N\int_\zeta H^0_kR^\lambda_k\w 
(g\w  B)_{n-k}\w\Phi
$$
is analytic for $\Re\lambda>0$ and tends to
$$
\int_{z\in Z}\psi(z)\w \K\Phi(z)
$$
when $\lambda\searrow 0$.
For $k\ge p$ we have, as before, cf., \eqref{modehatt}
that 
$$
\int_{z\in Z}\psi(z)\w\int_\zeta H^0_kR^\lambda_k\w (g\w  B)_{n-k}\w\Phi=
\int_\zeta R^\lambda_k \w\Phi\w T\psi,
$$
where  $T\psi(\zeta)$ is  continuous. 
If  $a_j=w''_j$ defines $Z$ locally as in the proof of Proposition~\ref{regular},
then $|F|\sim |a|$, and (see \cite{AW1})
$$
u_k=\alpha_k(u_p\oplus \alpha)
$$
where $\alpha, \alpha_k$ are smooth and $u_p$ is the form from Example~\ref{ulla}.
For $\Re\lambda>0$, the form $R_k^\lambda$ is locally integrable,
and in view of \eqref{anita} we have that $R^\lambda_k\to R_k$
as measures when $\lambda\searrow 0$. 
On the other hand, if  $1\le k<p$, then 
$$
T\psi(\zeta)=\int_{z\in Z}H^0_k\w (g\w B)_{n-k}\w\Psi(z)=\Ok(|a(\zeta)|^{-(2p-2k-1)}).
$$
Moreover, $u_k=\alpha_k(u_k\oplus \alpha)=\Ok(1/|a|^{2k-1})$.
Thus 
$$
\int_{z\in Z}\Psi(z)\int_\zeta 
 H^0_kR^\lambda_k(g\w B)_{n-k} \w\Phi=
\int_\zeta \Ok(\lambda |a|^{2\lambda-2p+1})
$$
which tends to $0$ when $\lambda\to 0$. 
Finally, the case $k=0$ is handled  by dominated
converence. %%%$R^\lambda_0\to 0$ b
%%%%
The second integral in \eqref{august1} is treated  in a similar way.
\end{proof}

\begin{lma}\label{brum}
Let $\Phi$ be a non-negative function in $\R^N_x\times\R^N_y$ such that
$\Phi^2$ is smooth and $\sim|x-y|^2$. For each integer  $\ell\ge 0$, let
$\alpha_\ell$ denote a smooth function that is $\Ok(|x-y|^\ell)$, and let
$\E_\nu$ denote a finite sum 
$
\sum_{\ell\ge 0}\alpha_\ell/\Phi^{\nu+\ell}.
$
If $\nu<N$ and $\xi\in C_c^k(\R^N)$, then
$$
T\xi(x)=\int_y\E_\nu(x,y)\xi(y)dy
$$
is in $C^k(\R^N)$.
\end{lma}

%%Assume that $\Phi>0$ and $\Phi^2$ is smooth
%%
%%Suppose that $V$ is a smooth subvariety of dimension $m\ge 2$ 
%%of an open set $\Omega\subset\R^N$ 
%%and  that $\xi\in C^k(V)$ with compact support. 
%%Assume that  $\eta_j=x_j-y_j$.
%%$\eta_1,\ldots,\eta_N$ define the diagonal in $V\times V$.
%%Then  %%
%%$$
%%h(x)=\int_{y\in V} \frac{\eta_i \xi(y)dS(y)}
%%{|\eta|^{m}}
%%$$
%%is in $C^k(V)$ as well for $i=1,\ldots,N$.
%%\end{lma}

%%Notice that  $h(x)$ is continuous  but {\it not}  smooth 
%%in $\Omega$ in general.  
%%
This lemma should be well-known, but for the reader's convenience we sketch
a proof. Let
$L_j=(\partial/\partial x_j+\partial/\partial y_j)$. 
It is  readily checked that  $L_k\alpha_\ell=\alpha_\ell$ from which we conclude that
$L_k\E_\nu=\E_\nu$. The lemma then follows.

%%\begin{proof}
%%Notice that $h(x)=(\partial/\partial x_i)g(x)$  where
%%$$
%%g(x)=c\int_V \frac{1}{|\eta|^{m-2}}\xi(y)dS(y),
%%$$
%%with  the convention that $1/|\ |^0=\log|\ |$.
%%Now take a smooth  vector field 
%%$
%%\L_x=a(x)\cdot(\partial/\partial x)
%%$
%%that is tangential to $V$ and let
%%$$
%%\tilde\L=a(x)\cdot\frac{\partial}{\partial x}+a(y)\cdot\frac{\partial}{\partial y}.
%%$$
%%Let $\alpha_\ell(x,y)$ denote any smooth function on $\R^N\times \R^N$
%%that is $\Ok(|\eta|^{\ell})$ and let  $\E_\nu$ denote any  kernel 
%%that is of the form $\alpha_\ell/|\eta|^{\nu+\ell}$
%%for some $\ell\ge 0$. 
%%One can verify that 
%%$\tilde\L\alpha_\ell=\alpha_\ell$
%%and it follows that 
%%$\tilde \L\E_\nu=\E_\nu.$
%%If $\phi(y)$ is in $C^1(V)$ we then have, since $\L$ is tangential, 
%%$$
%%\L_x\int_V\E_{m-2}\phi(y)dS(y)=
%%\int_V\E_{m-2}\phi(y)dS(y)+\int_V\E_{m-2}\L^*_y\xi(y)dS(y),
%%$$
%%where $\L^*$ is the formal adjoint of $\L$. By repeated use of this
%%equality %%of \eqref{spett}
%%it follows that we can apply $k$ such vector fields fo $g$ if
%%$\xi$ is in $C^k(V)$.
%%Since 
%%$$
%%\L_x  \frac{\partial}{\partial x_i}=
%%\frac{\partial}{\partial x_i}\L_x+\sum_k b_k \frac{\partial}{\partial x_k}
%%$$
%%it follows by induction and via a simple regularization of the kernel that 
%%all weak (tangential) derivatives of $h$ up to order $k$ are continuous.
%%From this the lemma follows.
%%\end{proof}

\section{Proofs  of Theorems~\ref{main} and \ref{mainglobal}}\label{pfs}

\begin{proof}[Proof of Theorem~\ref{main}]
If we  choose $g$ as the weight from Example~\ref{alba} then 
$\Pr \phi$ will vanish for degree reasons unless $\phi$ has bidegree
$(0,0)$, i.e., is a function, and in that case clearly
$\Pr \phi$ will be holomorphic for all $z$ in a $\omega$.
Now Theorem~\ref{main} follows from the Koppelman formula
\eqref{koppelman3} except for the asymptotic estimate
\eqref{asymp}.

After a slight regularization we may assume that
$\delta(z)$ is smooth on $Z_{reg}$ or alternatively we can replace 
$\delta$ by  $|h|$ where 
$h$ is a tuple of functions in $\Omega$ such that $Z_{sing}=\{h=0\}$,
by virtue of Lojasiewicz' inequality,
\cite{Loja} and \cite{Ma}.
%%%%
Let $\mu=HR$.  We have to estimate %%(assuming $\eta=\zeta-z$ for simplicity)
\begin{equation}\label{puck}
\int_\zeta \mu(\zeta)\frac{\Ok(|\eta|)}{|\eta|^{2n-2p}}
\end{equation}
when $z\to Z_{sing}$.
To this end we take a smooth approximand of $\chi_{[1/\sqrt{2},\infty)}(t)$
and write \eqref{puck}  as
\begin{multline*}
\int_\zeta \chi(\delta(\zeta)/\delta(z))
\mu(\zeta)\frac{\Ok(|\eta|)}{|\eta|^{2n-2p}}+ 
\int_\zeta\big(1-\chi(\delta(\zeta)/\delta(z))\big)
\mu(\zeta)\frac{\Ok(|\eta|)}{|\eta|^{2n-2p}}.%%%=I+II.
\end{multline*}
In the first integral $\delta(\zeta)\sim\delta(z)$ and since
\eqref{gammak} holds here and the integrand is
integrable we can use \eqref{asympupp} and get the estimate
$\lesssim \delta(z)^{-M}$ for some $M$.
%%$|\gamma(\zeta)|\lesssim |h(z)|^{-M}$ and
%%the integrand is locally integrable so we get that estimate.
In the second integral we use instead that
$\mu$ has some fixed finite order so that the action can be estimates by
a finite number of derivatives of
$
(1-\chi)\Ok(|\eta|)/|\eta|^{2n-2p},
$
which again is like $\delta(z)^{-M}$ for some $M$,  since
here 
$
C|\eta|\ge|\delta(\zeta)-\delta(z)|\ge \delta(z)/2.
$
Thus we have 
$
|\K\phi(z)|\lesssim \delta(z)^{-M}.
$
%%Thus Theorem~\ref{main} is proved.
\end{proof}

\begin{proof}[Proof of Corollary~\ref{maincor}]
Suppose that $\nu$ is the order of the current $R$.
Since $\K\Phi$ essentially is the current $R$ acting on
$\Phi$ times a smooth form, it is clear that the Koppelman
formula  remains  true even if $\Phi$ is just of class
$C^{\nu+1}$ in a \nbh of $Z$. However, it seems to be 
more delicate matter to check that 
$\K\Phi$ only depends on the pullback of $\Phi$ to $Z$. In order to
copy the argument in the proof of Proposition~\ref{regular}
one may need (possibly just for technical reasons) some
more regularity.   After appropriate 
resolutions of singularities, the current  $R$ 
is (locally) the push-forward of a finite sum of simple current of the form 
$$
\dbar\frac{1}{t_1^{a_1}}\w\ldots\w\dbar\frac{1}{t_r^{a_r}}
\w\frac{\alpha}{t_{r+1}^{a_{r+1}}\cdots t_m^{a_m}},
$$
where $\alpha$ is smooth.
If we choose $N$ as the sum of the powers
of the denominators then the argument will work. 
This follows from an inspection of the arguments
in \cite{AW2} but we omit the  details.
In general, however,  the number $N$  is much higher than the order of $R$.
%%as is illustrated  by the examples in Section~\ref{exsec}.
%%In any case, for $N$ large enough the corollary follows. 
\end{proof}

We now turn our attention to the proof of Theorem~\ref{mainglobal}.
We first assume that $X=Z$ is a subvariety of some domain $\Omega$ in $\C^n$.
A basic problem with the globalization is that we cannot assume that there
is one single resolution \eqref{acomplex} of $\Ok/\J$ in the whole
domain $\Omega$. We therefore must patch together local solutions. To this end we will
use Cech cohomology. Recall that if $\Omega_j$ is an open cover of $\Omega$, then
a $k$-cochain $\xi$ is a formal sum
$$
\xi=\sum_{|I|=k+1}\xi_I\w\epsilon_I
$$
where $I$ are multi-indices and $\epsilon_j$ is a nonsens basis, cf., e.g.,
\cite{A1} Section~8. Moreover, in this language
the coboundary operator $\rho$ is defined as
$\rho \xi = \epsilon\w \xi$,
where $\epsilon=\sum_j \epsilon_j$.

If $g$ is a weight as in Example~\ref{alba}  and $g'=(1-\chi)s/\nabla_\eta s$, then
\begin{equation}\label{gprim}
\nabla_\eta g'=1-g.
\end{equation}

Notice that the relations \eqref{Hdef} for the  Hefer morphism(s) 
can be written simply as
$$
\delta_\eta H=Hf-f(z)H=Hf
$$
if $z\in Z$.

\begin{proof}[Proof of Theorem~\ref{mainglobal} in case $Z\subset\Omega\subset \C^n$] 
Let $\Omega_j$ be a locally finite  open cover of $\Omega$  with convex polydomains
(Cartesian products of convex domains in each variable), and for each $j$ let 
$g_j$ be   a  weight with support in a slightly larger convex
polydomain  $\tilde\Omega_j\supset\supset\Omega_j$
and holomorphic in $z$ in a \nbh  of $\overline\Omega_j$.
Moreover, for each $j$ suppose that we have a given resolution \eqref{acomplex}
in $\tilde\Omega_j$, 
choice of Hermitian metric, a choice of Hefer morphism, and let  $(HR)_j$ be   the
resulting current.  If $\phi$ is a $\dbar$-closed $(0,q)$-form in $\Omega$, then
\begin{equation}\label{locsol}
u_j(z)=\int \big((HR)_j\w g_j\w B\big)_n\w\phi
\end{equation}
is a solution in $\Omega_j$ to $\dbar u_j=\phi$.
We will prove that $u_j-u_k$ is (strongly) holomorphic on $\Omega_{jk}\cap Z$
if $q=1$ and $u_j-u_k=\dbar u_{jk}$ on $\Omega_{jk}\cap Z_{reg}$ if
$q>1$, and more generally:

\smallskip
{\bf Claim I\ } {\it Let $u^0$ be the $0$-cochain $u^0=\sum u_j\w\epsilon_j$. For each
$k\le q-1$ there is  a $k$-cochain of $(0,q-k-1)$-forms on $Z_{reg}$ such that
$\rho u^k=\dbar u^{k+1}$ if $k<q-1$ and  $\rho u^{q-1}$ is a 
(strongly) holomorphic $q$-cocycle.} 
\smallskip

The holomorphic $q$-cocycle $\rho u^{q-1}$ defines a class in $H^q(\Omega,\Ok/\J)$ and
if $\Omega$ is pseudoconvex this class  must vanish,
i.e., there is a holomorphic $q-1$-cochain $h$ such that $\rho h=\rho u^{q-1}$.
By standard arguments this yields a global solution to $\dbar\psi=\phi$. For instance,
if $q=1$ this means that we have holomorphic functions $h_j$ in $\Omega_j$ such that
$u_j-u_k=h_j-h_k$ in $\Omega_{jk}\cap Z$. It follows that $u_j-h_j$ is a global
solution in $Z_{reg}$. 
%%\end{proof}

\smallskip%% \begin{proo

We thus have to prove Claim~I.
To begin with we assume that we have a fixed resolution with a fixed metric 
and Hefer morphism;  thus 
a fixed choice  of current $HR$.  Notice that if
$$
g_{jk}=g_j\w g_k'-g_k\w g_j',
$$
cf., \eqref{gprim}, then 
$$
\nabla_\eta g_{jk}=g_j-g_k
$$
in $\tilde\Omega_{jk}$. With $g^\lambda$ as in Section~\ref{koppsec}, 
and in view of \eqref{bm},  we have 
$$
\nabla_\eta(g^\lambda\w g_{jk}\w  B)=
g^\lambda\w g_j\w B-g^\lambda\w g_k\w B - g^\lambda\w  g_{jk} 
+ g^\lambda\w g_{jk}\w [\Delta].
$$
However, the last term must vanish since $[\Delta]$ has full degree in
$d\eta$ and $g_{jk}$ has at least degree $1$. Therefore
$$
-\dbar(g^\lambda\w g_{jk}\w  B)_n= (g^\lambda\w g_j\w B)_n-(g^\lambda\w g_k\w B)_n -(g^\lambda\w g_{jk})_n
$$
and as before we can take $\lambda=0$ and get, 
assuming  that $\dbar\phi=0$,
\begin{equation}\label{skillnad}
u_j-u_k=
%%\int (HR\w g\w B)_n\w\phi-\int (HR\w g'\w B)_n\phi=\\
\int (HR\w g_{jk})_n\w \phi +\dbar_z\int (HR\w g_{jk}\w B)_n\w \phi.
\end{equation}
Since $g_{jk}$ is holomorphic in $z$ in $\Omega_{jk}$ it follows that
$u_j-u_k$ is  (strongly)  holomorphic in $\Omega_{jk}\cap Z$ if $q=1$ and
$\dbar$-exact on $\Omega_{jk}\cap Z_{reg}$  if $q>1$.
%%

%%Now suppose that we have a fixed resolution but different choices of metrics
%%and Hefer forms in $\Omega_j$, and thus different $a_j=(HR)_j$.

\smallskip
{\bf Claim II\ }
{\it Assume that we have a fixed resolution but different choices of
Hefer forms and metrics and thus different
$a_j=(HR)_j$ in $\tilde\Omega_j$. Let $\epsilon'_j$ be a nonsense basis.
If
$A^0=\sum a_j\w\epsilon'_j$, then for each $k>0$ there is a $k$-cochain 
$$
A^k=\sum_{|I|=k+1}A_I\w\epsilon_I', 
$$
where  $A_I$ are currents on $\tilde\Omega_I$ with support
on $\tilde\Omega_I\cap Z$ and holomorphic in $z$ in $\Omega_I$, such that
\begin{equation}\label{bus3}
\rho'A^k=\epsilon'\w A^k=\nabla_\eta A^{k+1}.
\end{equation}}
\smallskip

In particular we have currents 
$a_{jk}$  with support on $Z$ and such that
$
\nabla_\eta a_{jk}=a_j-a_k
$
in $\tilde\Omega_{jk}$.
If
$$
w_{jk}=a_{jk}\w g_j\w g_k+a_j\w g_j\w g_k'-a_k\w g_k\w g_j',
$$
then
$$
\nabla_\eta w_{jk}=a_j\w g_j-a_k\w g_k.
$$
Notice that $w_{jk}$ is a globally defined current.
By a similar argument as above (and via a suitable limit process) one gets that
$$
u_j-u_k=
\int(w_{jk})_n\w\phi+\dbar_z\int(w_{jk}\w B)_n\w \phi
$$
in $\Omega_{jk}$ as before. In general  we put 
$$
\epsilon'=g=\sum g_j\w\epsilon_j.
$$ 
If, cf.,\eqref{gprim},
$$
g'=\sum g'_j\w\epsilon_j
$$
then
$$
\nabla_\eta g'=\epsilon-g=\epsilon-\epsilon'.
$$
If $a_I$ is a form on $\tilde\Omega_I$, then
$a_I\w\epsilon'_I$ is a well-defined global form. 
Therefore $A$ and hence 
$$
W=A\w e^{g'},
$$ 
(i.e., $W^k=\sum_j A^{k-j} (g')^{j}/j!$)
has globally defined coefficients and
$$
\rho W=\nabla_\eta W.
$$
In fact, since $A$ and $g'$ have even degree, 
$$
\nabla_\eta(A\w e^{g'})=\epsilon'\w A\w e^{g'}+
A\w e^{g'}\w (\epsilon-\epsilon')=\epsilon\w A\w e^{g'}.
$$
By the yoga above  then the $k$-cochain 
$$
u^k=\int (W^k\w B)_n\w\phi
$$
satisfies
$$
\rho u^k=\dbar_z \int (W^{k+1}\w B)_n\phi +\int(W^{k+1})_n\w\phi.
$$
Thus  $\rho u^k=\dbar u^{k+1}$ for $k<q-1$ whereas $\rho\w u^{q-1}$
is a holomorphic $q$-cocycle as desired.

\smallskip
It remains to consider the case when we have different resolutions in $\Omega_j$.
For each pair $j,k$
choose a weight $g_{s_{jk}}$ with support in $\tilde\Omega_{jk}$ that is 
holomorphic in $z$ in $\Omega_{s_{jk}}=\Omega_{jk}$.
By Theorem~3 Ch.~6 Section~F in \cite{GR}
we can choose a resolution in $\tilde\Omega_{s_{jk}}=\tilde\Omega_{jk}$ 
in which  both of the
resolutions in $\tilde\Omega_j$ and $\tilde\Omega_k$ restricted to $\Omega_{s_{jk}}$ 
are direct summands.
Let us fix metric and Hefer form and thus a current $a_{s_{jk}}=(HR)_{s_{jk}}$ 
in $\Omega_{s_{jk}}$
and thus a  solution $u_{s_{jk}}$ corresponding to $(HR)_{s_{jk}}\w g_{s_{jk}}$.
If we extend the metric and Hefer form from $\tilde\Omega_j$ in a way that respects
the direct sum, then $(HR)_j$ with these extended choices will be unaffected,
cf., Section~4 in \cite{AW1}. On $\tilde\Omega_{js_{jk}}$ we
therefore  practically speaking 
have  just one single resolution and as before thus  $u_j-u_s$
is holomorphic (if $q=1$) and $\dbar  u_{j s_{jk}}$ if $q>1$.
It follows that 
$u_j-u_k=u_j-u_s+u_s-u_k$ is holomorphic on $\Omega_{jk}$ if $q=1$ and equal to
$\dbar$ of 
$$
u_{jk}=u_{js_{jk}}+u_{s_{jk}k}
$$
if $q>1$.
%% and $\dbar$-exact
%%on $\Omega_{jk}\cap Z_{reg}$ if $q>1$.
We now claim that each $1$-cocycle 
\begin{equation}\label{bart}
u_{jk}+u_{kl}+u_{lj}
\end{equation}
is holomorphic on $\Omega_{jkl}$ if $q=2$ and
$\dbar$-exact on $\Omega_{jkl}\cap Z_{reg}$ if $q>2$.
%%Again this is clear from the first part of the proof
%%if all index belong to the same resolution.
On $\tilde\Omega_{s_{jkl}}=\tilde\Omega_{jkl}$ we can choose a resolution in which
each of the resolutions associated with the indices $s_{jk}, s_{kl}$ and $s_{kj}$
are direct summands. It follows that 
$
u_{js_{jk}}+u_{s_{jk}s_{jkl}}+u_{s_{jkl}j}
$
is holomorphic if $q=2$ and $\dbar u_{j s_{jk}s_{jkl}}$ if $q>2$.
Summing up, the statement about \eqref{bart} follows.
If we continue in this way Claim~I follows.

\smallskip
It remains to prove Claim~II.
It is not too hard to check by an appropriate
induction procedure, cf., the very construction
of Hefer morphisms in \cite{A7}, that if we have two choices
of (systems of) Hefer forms $H_j$ and $H_k$ for the same
resolution $f$, then there is a form
$H_{jk}$ such that %%(assuming $z\in Z$),
\begin{equation}\label{tarta1}
\delta_\eta H_{jk}=H_j-H_k+f(z)H_{jk}-H_{jk}f.
\end{equation}
More generally, if
$$
H^0=\sum H_j\w\epsilon_j
$$
then for each $k$ there is a  (holomorphic) $k$-cochain $H^k$ such that
(assuming $f(z)=0$ for simplicity)
\begin{equation}\label{bus1}
\delta_\eta H^k=\epsilon\w H^{k-1}-H^k f
\end{equation}
(the difference in sign between \eqref{tarta1} and \eqref{bus1} is because
in the latter one $f$ is to the right of the basis elements).

Elaborating the construction in Section~4 in \cite{AW1}, cf., 
Section~8 in \cite{A1},  one  finds, given $R^0 
=\sum R_j\w\epsilon_j$,  $k$-cochains of currents $R^k$ such that 
\begin{equation}\label{bus2}
\nabla_f R^{k+1}= \epsilon\w  R^k.
\end{equation}
%%

%%For forms $A,B$ we  define a product in the following way: 
We now define a product of forms in the following way.
If  the multiindices $I,J$ have no index in common, then
$(\epsilon_I,\epsilon_J)=0$, whereas
$$
(\epsilon_I\w\epsilon_\ell,\epsilon_\ell\w\epsilon_J)=\frac{|I|!|J|!}{(|I|+|J|+1)!}
\epsilon_{I}\w\epsilon_J.
$$
We then extend it to any forms bilinearly in the natural way.
It is easy to check that 
$$
(H^k f, R^\ell)=-(H^k, f R^\ell).
$$
Using  \eqref{bus1} and \eqref{bus2} (and keeping in mind that
$H^k$ and $R^\ell$ have odd order) one can verify that 
$$
\nabla_\eta(H^k,R^\ell)=(\epsilon\w H^{k-1}, R^\ell)
+(H^k, \epsilon\w R^\ell).
$$
By a similar argument one can  finally check  that
$$
A^k=\sum_{j=0}^k(H^j,R^{k-j})
$$
will satisfy \eqref{bus3}.
%%If $|A|$ denote the total degree of $A$, then one easily checks that
%%$$
%%(B,A)=(-1)^{(|A|+1)(|B|+1)}(A,B).
%%$$
%%It follows that
%%$$
%%\nabla_\eta (A,B)=(\nabla_\eta A,B)+(-1)^{|A|+1}(A,\nabla_\eta B).
%%$$
%%Since the mapping $f$ is odd it follows that
%%$$
%%(Af,B)=-(A,fB).
%%%%%%%%%%%%%%%%%%%%%
Thus  Claim~II and  hence Theorem~\ref{mainglobal} is proved in case 
$Z=X$ is a subvariety of $\Omega\subset\C^n$. 
\end{proof}

The extension to a general analytic space $X$ is done in pretty much the same
way and we just sketch the basic  idea. First assume that we have
a fixed $\eta$ as before but two different choices $s$ and $\tilde s$ of
admissible form, and let $B$ and $\tilde B$ be the corresponding locally
integrable forms. Then, see \cite{A13}, 
\begin{equation}\label{bbprim}
\nabla_\eta (B\w\tilde B)=\tilde B-B
\end{equation}
in the current sense, 
and by a minor modification of Lemma~\ref{brum}
one can  check that 
$$
\int (HR\w g\w B\w \tilde B)_n\w\phi
$$
is smooth on $X_{reg}\cap\omega$; for degree reasons it  vanishes if $q=1$.
It follows from \eqref{bbprim} that 
$
\nabla_\eta(HR^\lambda\w g\w B\w \tilde B)=
HR^\lambda\w g\w \tilde B-
HR^\lambda\w g\w B
$
from which we can conclude that
\begin{multline}\label{hast}
\dbar_z\int (HR\w g\w B\w\tilde B)_n\w\phi=\\
\int(HR\w g\w B)_n\w\phi-\int(HR\w g\w \tilde B)_n\w\phi,
\quad  z\in \omega\cap Z_{reg}.
\end{multline}

Now  let us assume  that
we have two local solutions,  in say $\omega$ and $\omega'$,  obtained from
two different embeddings of slightly larger sets $\tilde \omega$ and $\tilde\omega'$
in subsets of $\C^{n}$ and $\C^{n'}$, respectively. We want to compare these
solutions on $\omega\cap\omega'$. 
Localizing further, as before, we may assume that the weights %%that are used
both have support in $\tilde \omega\cap\tilde \omega'$.  After adding nonsense
variables  we may assume that both embeddings are into the same
$\C^n$, and after further localization there is a local
biholomorphism in $\C^n$ that maps one embedding onto the other one,
see \cite{GR}.  
(Notice that a solution  obtained via an embedding in $\C^{n_1}$
also can be obtained via an embedding into a larger $\C^n$, by just adding
dummy variables in the first formula.)
In other words, we may assume that we have the same
embedding in some open set $\Omega\subset\C^n$ but two solutions obtained  from
different  $\eta$ and $\eta'$. 
(Arguing as before, however, we may assume that we have the same resolution
and the same residue current $R$.)
Locally there is an  invertible matrix $h_{jk}$ such that
\begin{equation}\label{arvika}
\eta'_j=\sum h_{jk}\eta_k.
\end{equation}
We define a vector bundle mapping
$
\alpha^*\colon \Lambda_{\eta'}\to\Lambda_\eta
$
as  the identity on $T^*_{0,*}(\Omega\times\Omega)$
and so that  
$$
\alpha^* d\eta_j'=\sum h_{jk} d\eta_k.
$$
It is readily checked that
$$
\nabla_\eta\alpha^*= \alpha^*\nabla_{\eta'}.
$$
Therefore, $\alpha^*g'$ is an $\eta$-weight if $g'$ is an $\eta'$-weight.
Moreover,  if $H$ is an $\eta'$-Hefer morphism, then $\alpha^*H$
is an $\eta$-Hefer morphism, cf., \eqref{Hdef}. 
If $B'$ is obtained from an $\eta'$ admissible form $s'$, then
$\alpha^*s'$ is an $\eta$-admissible form and
$\alpha^*B'$ is the corresponding locally integrable form.
We  claim  that the $\eta'$-solution
\begin{equation}\label{prall}
v'=\int (H'R\w g'\w B')_n\w\phi
\end{equation}
is comparable to the $\eta$-solution
\begin{equation}\label{prall2}
v=\int \alpha^*(H'R)\w\alpha^* g'\w \alpha^* B'\w\phi.
\end{equation}
Notice that we are only interested in the $d\zeta$-component of the
kernels.
We have that ($d\eta=d\eta_1\w\ldots\w d\eta_n$ etc)
$$
(H'R\w g'\w B')_n=A\w d\eta' \sim
A\w\det(\partial\eta'/\partial\zeta) d\zeta
$$
and 
$$
\alpha^*(H'R\w g'\w B')_n=A\w\det h\w d\eta
\sim A\w\det h \det (\partial\eta/\partial\zeta) d\zeta.
$$ 
Thus
$$
\alpha^*(H'R\w g'\w B')_n\sim \gamma(\zeta,z)(H'R\w g'\w B')_n
$$
with
$$
\gamma=\det h \det\frac{\partial\eta}{\partial\zeta}
\Big(\det\frac{\partial\eta'}{\partial\zeta}\Big)^{-1}.
$$
From  \eqref{arvika} we have that
$\partial\eta'_j/\partial\zeta_\ell=
\sum_k h_{jk}\partial\eta_k/\partial\zeta_{\ell} +\Ok(|\eta|)
$
which implies that $\gamma$ is $1$ on the diagonal.
Thus  $\gamma$ is a smooth (holomorphic) weight and therefore 
\eqref{prall} and \eqref{prall2} are comparable, and thus the claim
is proved. This proves Theorem~\ref{mainglobal} in the case $q=1$, and
elaborating the idea  as in the
previous proof we obtain the general case.

\begin{remark}
In case $X$ is a Stein space and $X_{sing}$ is discrete there is a much simpler proof
of Theorem~\ref{mainglobal}. 
%%%
To begin with we can solve $\dbar v=\phi$ locally, and 
modifying by such local solutions we may assume that
$\phi$ is vanishing identically in a \nbh of $X_{sing}$.
There exists  a sequence of holomorphically convex open subsets
$X_j$ such that $X_j$ is relatively compact in $X_{j+1}$ and
$X_j$ can be embedded as a subvariety of some pseudoconvex set
$\Omega_j$ in $\C^{n_j}$. Let $K_\ell$ be the closure of $X_\ell$.
By Theorem~\ref{main} we can solve $\dbar u_\ell =\phi$ in a \nbh of 
$K_\ell$ and  $u_\ell$ will be smooth. 
If $q>1$ we can thus  solve $\dbar w_\ell=u_{\ell+1}-u_\ell$
in a \nbh of $K_\ell$, and since $Z_{sing}$ is discrete we can 
assume that $\dbar w_\ell$ is smooth in $X$.  Then
$v_\ell=u_\ell-\sum_1^{\ell-1}\dbar w_k$ defines a global solution.
If $q=1$, then  one  obtains a global solution in a similar way
by  a Mittag-Leffler type argument. 
\end{remark}

%%\section{Global analysis of the residue}
\section{The asymptotic estimate}\label{asympsec}

To catch the asymptotic behaviour we have to globalize 
the proof of the first part of Proposition~\ref{regular}.  
%%As before $\Omega$ is some open \nbh of the closed unit ball.

Since the functions $f_1^j$ generate the ideal $\J$, 
given any fixed point $x$ on $Z_{reg}$ we can extract
$h_1,\ldots, h_p$ from $f_1^j$ such that
$dh_1\w\ldots\w dh_p\neq 0$ at $x$. After a  reordering of the variables
we may  assume that $\zeta=(\zeta',\zeta'')=
(\zeta',\zeta''_1,\ldots,\zeta''_p)$ such that 
$H=\det(\partial h/\partial\zeta'')\neq 0$ at $x$. 
Outside the hypersurface $\{H=0\}$  we can (locally) make the change of coordinates
$(\omega',\omega'')=(\zeta', h(\zeta',\zeta''))$ since
$$
\frac{d(\omega',\omega'')}{d(\zeta',\zeta'')}=H.
$$
Moreover, 
$$
\frac{\partial}{\partial\omega_j''}=\frac{1}{H}\sum_k A_{jk}
\frac{\partial}{\partial\zeta''_k},
$$
where $A_{jk}$ are global holomorphic functions. 
Therefore, 
anywhere outside $\{H=0\}$ we have that
\begin{equation}\label{halmstra}
\dbar\frac{1}{h_p}\w\ldots\w\dbar\frac{1}{h_1}=
\frac{\det A_{jk}}{H^p}\otimes
\frac{\partial}{\partial\zeta''_1}\w\cdots\w\frac{\partial}{\partial \zeta''_p}
\lrcorner [Z].
\end{equation}
%%%where $\alpha_jk$ is a  global holomorphic matrix. 

\begin{prop}\label{allan}
Given a point $x\in Z_{reg}$, there is a hypersurface  $Y=\{H=0\}$ avoiding
$x$ such that
\begin{equation}\label{allaneq}
R_p=\tau \dbar\frac{1}{h_p}\w\ldots\w\dbar\frac{1}{h_1},
\end{equation}
where  $\tau$ is smooth outside $Y$ and $\tau=\Ok(H^{-M})$ for some
$M>0$.
\end{prop}

It follows from \eqref{halmstra} and \eqref{allaneq}, cf., the proof
of Proposition~\ref{hyp},  that
$$
|\gamma_p|\le C|H|^{-M}.
$$
With a finite number of such choices $H_j$ we have that
$$
Z_{sing}=\cap_j\{H_j=0\}
$$
and thus
$$
|\gamma_p(z)|\lesssim \min_j |H_j(z)|^{-M_j}\le
C|H(z)|^{-M},
$$
where $H=(H_1,\ldots,H_\nu)$. However $|H|\ge\delta^N$ for some $N$ and hence
\eqref{asympupp} follows for $k=p$.

%%Thus the estimate \eqref{asymp} follows. %%proof is complete.

It remains to prove  Proposition~\ref{allan}.
We begin with the following simple lemma.

\begin{lma}\label{simple}
Assume that $F_1,\ldots,F_m,\Phi$ are holomorphic $r$-columns at $x\in \Omega$
and that the germ $\Phi_x$ is in the submodule of $\Ok_x^{\oplus r}$  generated by 
$(F_j)_x$. If $F_j,\Phi$ have meromorphic extensions to $\Omega$, then there
are holomorphic $A_j$ with meromorphic extension to (a possibly somewhat smaller 
neighborhood)  $\Omega$
such that $\Phi=A_1F_1+\cdots +A_mF_m$.
\end{lma}

\begin{proof} The analytic sheaf
$\F=(F_1,\ldots,F_m,\Phi)/(F_1\ldots,F_m)$ is coherent in $\Omega$ and vanishing at
$x$ so it must have support on a variety $Y$ outside $x$. If $h$ is holomorphic
and vanishing on $Y$, then $h^M\F=0$ in a Stein \nbh $\Omega'$ of the closed ball if
$M$ is large enough. Therefore there are holomorphic functions $a_j$ 
in $\Omega'$ such that $h^M\Phi=\sum a_j F_j$.
\end{proof}

Suppose that the holomorphic $r$-columns $F=(F_1,\ldots, F_m)$ and
$\tilde F=(\tilde F_1,\ldots, \tilde F_{\tilde m})$ are minimal
generators of  the same sheaf at $x$.
It is well-known that then   $\tilde m=m$ and there is a holomorphic
invertible $m\times m$-matrix $a$ at $x$ such that $\tilde F=Fa$.

\smallskip

\noindent{\bf Claim I}
{\it If $F,\tilde F$ have meromorphic extensions to $\Omega$, 
then we may assume that $a$ has as well.} 

\begin{proof}
By Lemma~\ref{simple} we have global
meromorphic matrices $a$ and $b$, holomorphic at $x$, such that
$\tilde F=Fa$ and $F=\tilde F b$. Thus $F=Fab$, and since $F$
is minimal, it follows that $ab=I+\alpha$ where the entries in $\alpha$ belong to the
maximal ideal at $x$, i.e., $\alpha(x)=0$. Therefore the matrix
$I+\alpha$ is invertible at $x$, and so $b(I+\alpha)^{-1}$ is
a meromorphic inverse to $a$ that is holomorphic and an 
isomorphism at $x$.
\end{proof}

Assume that  $\F$ is a coherent sheaf in $\Omega$ of  codimension
$p$ at $x$ and let  
$\Ok(E_k), f_k$ and $\Ok(\tilde E_k),\tilde f_k$, $k=0,\ldots,p$,
be  two minimal free resolutions of $\F$ at $x\in \Omega$. 
Moreover, assume that all $f_k,\tilde f_k$ have meromorphic extensions
to $\Omega$. By iterated use of Claim~I we get:

\smallskip
\noindent{\bf Claim II}
{\it There are  isomorphisms $g_k\colon \Ok(E_k)\to\Ok(\tilde E_k)$
holomorphic at $x$ and with meromorphic extensions to $\Omega$
such that $g_{k-1}f_k=\tilde f_k g_k$}

\smallskip

Assume for simplicity that $E_0=\tilde E_0$. Outside some hypersurface
$Y$ all  the mappings $f_k, \tilde f_k, g_k$ are holomorphic,
and there  we have well-defined currents $R_p$ and $\tilde R_p$,
and $\tilde R_p=g_p R_p$ there, cf., Section~4 in  \cite{AW1}.
Since the codimension  is $p$ and the complexes  end up at $p$ the residue
currents $R_p$ and $\tilde R_p$ are independent of the choice of 
Hermitian metrics. 

\smallskip
Now let  $\Ok(E_k), f_k$ be an arbitrary free resolution
of $\F$ in $\Omega$.  It is well-known that, given $x\in \Omega$,  there is locally a 
holomorphic decomposition $E_k=E_k'\oplus E_k''$, $f_k=f_k'\oplus f_k''$
such that $\Ok(E_k'), f_k'$ is a minimal free resolution of $\F$ at $x$
and $\Ok(E_k''), f_k''$ is a free resolution of $0$.
In other words, if we fix global holomorphic frames $e_k$
for $E_k$ to begin with, then there are holomorphic
$G_k$ with values in $GL(\rank E_k,\C)$ such that the first
$\rank E_k'$ elements  in  $e_kG_k$ generate $E_k'$
whereas the last ones generate $E''_k$. 
We claim, as the reader may expect at this stage, that 

\smallskip
\noindent{\bf Claim III}
{\it The $G_k$
can be assumed to have meromorphic extensions to $\Omega$.}

\begin{proof}[Proof of Claim~III]
We proceed by induction. Suppose that we have found the desired decomposition
up to $E_k$ and consider the mapping $f_{k+1}$ expressed in the new
frame of $E_k$ and the original frame for $E_{k+1}$. Thus (the matrix for)
$f_{k+1}$ is holomorphic at $x$ and globally meromorphic. Choose a minimal
number of columns of $f_{k+1}$ such that the restrictions to $E_k'$ generate
the stalk of $\Ker f_k$  at $x$. After a trivial reordering of the columns we may
assume that
$$
f_{k+1}=
\left(\begin{array}{cc}
f'_{k+1} &  \Phi' \\
\Psi  & \Phi'' 
\end{array}
\right)
$$
By Lemma~\ref{simple} there is a meromorphic matrix $a$, holomorphic and invertible
at $x$, such that 
$
\Phi'=f'_{k+1}a.
$
Therefore we can make the meromorphic change of frame
$$
\left(\begin{array}{cc}
f'_{k+1} &  \Phi' \\
\Psi  & \Phi'' 
\end{array}
\right)
\left(\begin{array}{cc}
I &  -a \\
0 & I 
\end{array}
\right)
=
\left(\begin{array}{cc}
f'_{k+1} &  0 \\
\Psi  & f''_{k+1} 
\end{array}
\right).
$$
We now claim that  
\begin{equation}\label{balla}
\Im f_{k+1}''=\Ker f_k''
\end{equation}
at $x$. By the lemma again we can then find a meromorphic matrix $a$,
holomorphic and invertible at $x$ such that 
$
\Psi=f_{k+1}''a,
$
and then after a similar meromorphic change of frame as before we get that
the mapping $f_{k+1}$ has the matrix
$$
\left(\begin{array}{cc}
f_{k+1}'&  0 \\
0 &  f_{k+1}''
\end{array}
\right)
$$
in the new frames.  Thus it remains to check \eqref{balla} which is indeed
a statement over the local ring $\Ok_x$ and therefore ``wellknown''. In any
case, for each  $z\in\Ker f_k''$ we can solve
$$
\left(\begin{array}{cc}
f'_{k+1} &  0 \\
\Psi  & f''_k 
\end{array}
\right)
\left(\begin{array}{c}
\xi\\
\eta  
\end{array}
\right)
=
\left(\begin{array}{c}
0\\
z  
\end{array}
\right).
$$
Since $f'_{k+1}$ is minimal this implies that $\xi$ is in the maximal ideal
at $x$ and hence $\Psi\xi$ is in the maximal ideal. Thus we can solve
$
f_{k+1}''=z-\alpha
$
with $\alpha$ in the maximal ideal for each $z\in\Ker f_k''$. However,
since $f''_\ell$ is a resolution of $0$ it follows that each
$\Ker f_k''$ is a free module. Expressed in a basis for $\Ker f_k''$
we can solve then
$
f_{k+1}''\eta=I-\alpha
$
and since $\alpha$ is in the maximal ideal it follows that
$I-\alpha$ is invertible; hence \eqref{balla} follows.
%%%%%
\end{proof}

We can now conclude the proof of Proposition~\ref{allan}.
Let us  equip the bundles  $E_k=E_k'\oplus E_k''$
with some  metrics that respect the decomposition, for instance the trivial
metric with respect to the ``new'' frame. Both 
$\Ok(E_k'), f_k'$ and the Koszul complex generated by $h$ are minimal
resolutions of $\F=\Ok/\J$ at $x$, and since  both of them have meromorphic extensions
to $\Omega$ by Claim~II  there is a meromorphic $g_p$, holomorphic  at $x$,  such that
$$
R'_p=g_p\dbar(1/h_p)\w\ldots\dbar(1/h_1).
$$
Here $R'_p$ is the current obtained from the resolution $f'_k$. 
If $\tilde R_p$ is the current with respect to the new metric, then
$$
\tilde R_p=
\left(\begin{array}{c}
g_p \\
0
\end{array}
\right)
\dbar\frac{1}{h_p}\w\ldots\w\dbar\frac{1}{h_1}
$$
with respect to the new frame, and hence we obtain the matrix for
$\tilde R_p$ with respect to the original frame after multiplying
with the matrix $G_p$. 
Notice that outside $Z_{p+1}$, the image of $f_{p+1}$ is a smooth (holomorphic)
subbundle
$H$ of $E_p$, and let $\pi$ be the orthogonal projection onto
the orthogonal complement (with respect to the original metric) of $H$.
Then, cf., \cite{AW1}, $R_p=\pi \tilde R_p$. Thus $\tau$ in
 \eqref{allaneq} is $\pi G_p (g_p\ 0)^T$, and since $\pi$ does not increase
norms,  the  estimate in Proposition~\ref{allan} follows.
%%Finally, $R_p$  is $\pi\tilde R_p$, where $\pi$ is the projection
%%(defined outside $Z_{p+1}$) onto the orthogonal complement (with respect to the original metric)
%%of the image of $f_{p+1}$. 
%%
%%smooth projection
%%defined outside $Z_{p+1}$, cf., \cite{AW1}.
%%
%%
%%obtained as a smooth global pointwise projection
%%of $\tilde R_p$. Hence $\tau$ in \eqref{allaneq} is essentially
%%the  meromorphic matrix $G_p$ followed by $\pi$ and hence
%%Proposition~\ref{allan} follows.
%%\tilde R_p=h(g_p \ 0)^T\dbar(1/g)
%%$$
%%wrt the original frame. where $h$ is meromorphic $r_p\times r_p$ matrix
%%holo and invertible at $x$.  Now 
%%$$
%%R_p=\pi \tilde R_p
%%$$
%%where $\pi$ is a certain global smooth projection, see \cite{AW1}, and
%%hence we can conclude that

\section{Examples}\label{exsec}

We explain what the currents $U$ and $R$ and 
our Koppelman formulas mean in the case of a reduced complete intersection.
We also illustrate the techniques of Section \ref{compsupp} where $\dbar$-closed extensions and 
solutions with compact support are considered.

Let $f_1,\ldots,f_p$ be holomorphic functions, defined in a 
suitable neighborhood of $\bar{\B}\subset \C^n$, and assume that $Z=\{f_1=\cdots=f_p=0\}$
has dimension $d=n-p$ and $df_1\w\ldots\w df_p\neq 0$ on $Z_{reg}$,
cf., Example~\ref{ulla}.
Then $R=R_p$ is given by \eqref{pdum} with $a$ replaced by $f$.

Let $h_j$ be Hefer $(0,1)$-forms so that
$\delta_{\eta}h_j=f_j(\zeta)-f_j(z)$ and let $\tilde{h}=\sum h_j\w e_j^*$;
recall that $\tilde{h}$ is a section of $\Lambda (A^* \oplus T^*(\Omega))$. 
The Hefer morphisms $H_k^{\ell}$ 
can be described as interior multiplication with $\tilde{h}^{k-\ell}/(k-\ell)!$ and
a straight forward computation shows that
\begin{equation*}
HR=H^0_pR_p=\dbar \frac{1}{f_p}\w \cdots \w \dbar \frac{1}{f_1}\w
h_1\w \cdots \w h_p= \gamma \lrcorner [Z]\w h,
\end{equation*} 
where $\gamma$ is a smooth $(p,0)$-vector field on $Z_{reg}$ such that
$\gamma \lrcorner df_p\w \cdots \w df_1=(2\pi i)^p$ and $h=h_1\w \cdots \w h_p$.
According to the proof of the Koppelman formula(s) above, our solution operator 
to $\dbar$ on $Z_{reg}$ is

\begin{equation}\label{ture}
\K\phi (z)=\int_{Z} \gamma \lrcorner [h \w (g\w B)_{d}]\w \phi. 
\end{equation}
and the projection operator is
\begin{multline}\label{gullan}
\Pr\phi = \int_{Z} \gamma \lrcorner [
h \w g_{d}]\w \phi 
%\,\,\,\,\,\,\,\,\,\,\,\,\,\,\,\,\,\,\,\,\, \\
%\,\,\,\,\,\,\,\,\,
= \int_{Z} \gamma \lrcorner \Big[h \w \frac{\bar{\zeta}\cdot d\zeta \w 
(d\bar{\zeta}\cdot d\zeta)^{d-1}}{(2\pi i(|\zeta|^2-z\cdot \bar{\zeta}))^{d}}\Big]
\w \dbar \chi \w \phi.
\end{multline}
Here  $g$ is the weight $g=\chi(\zeta)-\dbar \chi(\zeta)\w (\sigma/\nabla_{\eta}\sigma)$
from Example 2 and $B$ is the Bochner-Martinelli form associated with $\eta=\zeta-z$. 

In particular, the right hand side of \eqref{gullan} is a quite simple
representation formula for a strongly holomorphic function $\phi$ on $Z$. 

If $Z_{sing}$ is discrete, avoids the boundary, $\partial\B$, of the ball,
and $Z$ intersects $\partial\B$ transversally, then
we get back the representation formula for strongly holomorphic functions
of  Stout \cite{Stout} and Hatziafratis \cite{Hatzia} since then we may let $\chi$
in \eqref{gullan} be 
the characteristic function for $\B$ and the integral becomes an 
integral over $Z\cap \partial \B$. 

Suppose in addition that $d=1$. Let $\xi=\sum\xi_jd\eta_j$ be a form satisfying 
$\delta_{\eta}\xi=1$ outside $\Delta$,
e.g., $\xi= B_1$ or $\xi=\sigma$. For some function $C(z,\zeta)$ (a priori depending on $\xi$)
we have 
$h\w \xi=Cd\eta_1\w\cdots\w d\eta_n$. Applying $\delta_{\eta}$ to this equality we get
$(-1)^{n-1}h=C\delta_{\eta}(d\eta_1\w\cdots\w d\eta_n)$ for $(z,\zeta)\in Z\times Z$ since
$\delta_{\eta}h=0$ for such $(z,\zeta)$. From this we read off that $C\mid_{Z\times Z}$ is meromorphic,
independent of $\xi$, and 
with (at most) a first order singularity along the diagonal. We conclude that
$h\w (g\w B)_1=\chi Cd\eta_1\w\cdots\w d\eta_n$ and 
$h\w g_1=\dbar \chi \w Cd\eta_1\w\cdots\w d\eta_n$ on $Z\times Z$ and our solution kernels
$K$ and $P$ become
\begin{equation*}
K(z,\zeta)= \chi(\zeta)C(z,\zeta)\cdot(\gamma \lrcorner d\zeta),\quad
P(z,\zeta)= \pm \dbar \chi(\zeta) \w C(z,\zeta) \cdot (\gamma \lrcorner d\zeta).
\end{equation*} 
Notice that $\gamma \lrcorner d\zeta$ is a holomorphic $1$-form on $Z_{reg}$ since 
$\gamma \lrcorner [Z]=\dbar (1/f)$ there.
If $Z$ is the cusp $Z=\{f(z)=z_1^r-z_2^s=0\}\subset \C^2$,
where $r$ and $s$ are relativly prime integers $2\leq r<s$, one readily checks that 
\begin{equation*}
h=\frac{1}{2\pi i}\big(\frac{\zeta_1^r-z_1^r}{\zeta_1-z_1}d\eta_1-
\frac{\zeta_2^s-z_2^s}{\zeta_2-z_2}d\eta_2\big),
\quad
\frac{\gamma}{2\pi i}=\frac{r\bar{\zeta}_1^{r-1}\partial/\partial\zeta_1 -
s\bar{\zeta}_2^{s-1}\partial/\partial\zeta_2}{r^2|\zeta_1|^{2(r-1)}+s^2|\zeta_2|^{2(s-1)}}.
\end{equation*}
Using the paramertization $\tau \mapsto (\tau^s,\tau^r)=(\zeta_1,\zeta_2)$ of $Z$, 
a straight forward computation shows that 
$\gamma \lrcorner d\zeta_1\w d\zeta_2=2\pi id\tau/\tau^{(r-1)(s-1)}$,
yielding the following Cauchy formula 
\begin{equation*}
\phi(t)=\int_{|\tau|=\rho}\frac{\phi(\tau)C(\tau,t)d\tau}{\tau^{(r-1)(s-1)}}-
\lim_{\epsilon\to 0}\int_{\epsilon<|\tau|<\rho}\frac{\dbar \phi(\tau)\w C(\tau,t)d\tau}{\tau^{(r-1)(s-1)}},
\end{equation*}
on $Z$, where
\begin{equation*}
C(\tau,t)=\frac{1}{2\pi i}\frac{\tau^{rs}-t^{rs}}{(\tau^r-t^r)(\tau^s-t^s)}.
\end{equation*}
%%and $\rho$ is the positive root to $x^{2r}+y^{2s}=1$, 
%%i.e., so that $|\tau|=\rho$ corresponds to 
%%$Z\cap \partial \B$.

\medskip

Assume now,  cf., Section \ref{compsupp}, that
%%Let us also consider $\dbar$-closed extensions; cf., Section \ref{compsupp}. 
%%For simplicity we assume that 
$Z_{sing}\subset K \subset \subset \B$ and 
let $\varphi$ be a smooth $\dbar$-closed $(0,q-1)$-form
on $Z\setminus K$. Let $\chi$ and $\tilde{\chi}$ be cutoff functions in $\B$ such that $\chi$
 is $1$  in a neighborhood of $K$
and $\tilde{\chi}$ is $1$ in a neighborhood of $\textrm{supp}(\chi)$.  
Put
\begin{equation*}
\tilde{g}=\tilde{\chi}(z)-\dbar \tilde{\chi}(z)\wedge \sum_1^n
\frac{\bar{z}\cdot d\eta\wedge (d\bar{z}\cdot d\eta)^{k-1}}{(2\pi i(|z|^2-\bar{z}\cdot \zeta))^k},
\end{equation*}
i.e., $\tilde{g}$ is the weight from Example \ref{alba} with $z$ and $\zeta$ interchanged.
Our formulas show that \eqref{ture}, with $g$ replaced by $\tilde{g}$ and $\phi$ replaced by 
$\dbar \chi\w \varphi$,
is a solution with compact support in $\B$ 
(and in fact also smooth across $Z_{sing}$) 
to the equation $\dbar u=\dbar \chi \w \varphi$ on $Z_{reg}$
provided that the corresponding projection term, cf., \eqref{gullan},
\begin{equation}\label{kockum}
%%\int_Z \gamma \lrcorner h \w \tilde{g}_d \w \dbar\chi\w\varphi =
-\dbar \tilde{\chi}(z)\w \int_Z \gamma \lrcorner \Big[h \w
\frac{\bar{z}\cdot d\zeta\wedge (d\bar{z}\cdot d\zeta)^{d-1}}{(2\pi i(|z|^2-\bar{z}\cdot \zeta))^d}
\Big]\w  \dbar\chi\w\varphi.
\end{equation}
vanishes.
Then $(1-\chi)\varphi +u$ is smooth and $\dbar$-closed on $Z$, and coincides with $\varphi$ outside 
a neighborhood in $Z$ of $K$. 
As long as $q<d$, \eqref{kockum}
is trivally zero;   if $q=d$, then it is clearly 
sufficient that 
\begin{equation}\label{oscar}
\int_Z \dbar\chi\w \varphi\, \xi \w  (\gamma\lrcorner d\zeta)=0, \quad \xi\in \mathcal{O}(Z).
\end{equation} 
On the other hand, if $\varphi$ has a smooth 
$\dbar$-closed extension, then \eqref{oscar} holds. In particular we see that if $\varphi$ is 
holomorphic on the regular part of the cusp $Z$, then $\varphi$ is strongly holomorphic if and only if
\begin{equation*}
\int_{|\tau|=\epsilon}\varphi \xi d\tau/\tau^{(r-1)(s-1)}=0, \quad \xi \in \mathcal{O}(Z).
\end{equation*}

\section{Solutions formulas with weights}\label{weightsec}

For the proof of Theorem~\ref{main2} we use extra weight factors. 
Let $A$ be any subvariety of $Z$ that contains $Z_{sing}$, in particular
$A$ may be $Z_{sing}$ itself.  Let $a$ be a holomorphic tuple in $\Omega$ 
such that $\{a=0\}\cap Z=Z_{sing}$, and let  $H^a$ be a holomorphic $(1,0)$-form in $\Omega$
such that
$
\delta_\eta H^1=a(\zeta)-a(z).
$
If $\psi$ is a $(0,q)$-form that
vanishes in a \nbh of $Z_{sing}$  we can incorporate the weight  
\begin{equation}\label{viktig}
g_a^\mu=\Big(\frac{a(z)\cdot a}{|a|^2}+H^a\cdot\dbar\frac{\bar a}
{|a|^2}\Big)^\mu
\end{equation}
in \eqref{koppelman3}, i.e., we use  the weight $g_a^\mu\w g$
instead of just  $g$,  the usual weight with compact
support that is holomorphic in $z$. 
Since the  operators in  Lemma~\ref{brum}
are  bounded on $L^p_{loc}$, we have that 
\begin{equation}\label{sallan2}
\psi=\dbar \int_{Z_{reg}} \gamma\lrcorner (H\w g^\mu_a\w g\w B)_n\w \psi +
\int_{Z_{reg}} \gamma\lrcorner(H\w g^\mu_a\w g\w B)_n\w \dbar\psi,
\end{equation}
%%if $g$ depends holomorphically on $z$ and 
for $(0,q)$-forms $\psi$,  $q\ge 1$, in $L^p(Z_{reg})$
that vanish   in a \nbh of $Z_{sing}$. 
If $\phi$ is as in Theorem~\ref{main2},  thus
 \eqref{sallan2} holds for
$\psi=\chi(|a|^2/\epsilon)\phi$ for each $\epsilon>0$.
For each  natural  number $\mu$ we get a solution when $\epsilon\to 0$ 
in view of  the asymptotic estimate of $|\gamma|$ if just
$N$ is large enough.
If $\mu$ is large, then the solution will
vanish to high order at $Z_{sing}$ and therefore
Theorem~\ref{main2} follows.

\section{Solutions with compact support}\label{compsupp}

Theorems \ref{portensats}, \ref{main3}, and \ref{main3x} are Hartogs type theorems, because solvability
of $\dbar\psi=\phi$ in $X_{reg}$ roughly speaking means that $\psi$ has a $\dbar$-closed
smooth extension across $X_{sing}$. As usual therefore the proofs rely on the 
possibility to solve the $\dbar$-equation with compact support.

\smallskip
To begin with we assume that $Z$ is defined in a \nbh of the closed unit ball
$\overline\B$. Since the depth of $\Ok/\J$ is at least  $\nu$ we can choose,
see, e.g., \cite{Eis1},
a resolution \eqref{acomplex} with $N=n-\nu$,  and   the associated residue
current then is $R=R_p+\cdots +R_{n-\nu}$. Notice that $\dbar R_{n-\nu}=0$.

\begin{prop}\label{badanka2}
Let $Z$ be a subvariety of a \nbh of  $\overline \B$ with the single singular point $0$.  
Assume that $\phi$ is a smooth  $\dbar$-closed $(0,q)$-form
in  $Z\cap \B\setminus \overline{\B_\epsilon}$.
\smallskip

\noindent (i) If $q\le \nu-2$ there is a smooth $\dbar$-closed form $\Phi$ in 
$Z\cap \B$ that coincides with $\phi$ outside a \nbh in $Z$ of
$Z\cap\overline{\B_\epsilon}$.

\smallskip
\noindent (ii)  If $q=\nu-1$ the same is true if and only if 
\begin{equation}\label{valfrid}
\int R_{n-\nu}\w \dbar\chi\w h\phi \w d\zeta=
\int_Z\dbar\chi\w h \phi \w (\gamma_{n-\nu}\lrcorner d\zeta)=0,  \quad h\in\Ok(\B),
\end{equation}
if $\chi$ is a cutoff function in $\B$ that is $1$ in a \nbh
of $\overline{\B_\epsilon}$.
\end{prop}

Notice that  \eqref{valfrid}  holds for all such $\chi$ if it holds
for one single $\chi$. %%The equality 
%%If $Z$ intersects the boundary of $\B_\delta$ reasonably, then the condition
%%can be expressed as
%%\begin{equation}\label{pulver}
%%\int_{\partial\B_\delta}
%%R_{n-\nu}\w \phi\w h=\pm \int_{Z\cap\partial\B_\delta}
%%\gamma_{n-\nu}\lrcorner h\w\phi=0.
%%\end{equation}
%%
%%We  can rewrite the integrals in \eqref{valfrid} as  integrals over $Z$;
%%\begin{equation}\label{pulver}
%%\int
%%R_{n-\nu}\w \dbar\chi\w \phi h\w d\zeta=\pm \int_{Z}
%%\dbar\chi\w \gamma_{n-\nu}\lrcorner d\zeta\w \w  h\phi=0.
%%\end{equation}

\begin{proof}
First notice  that if $q=\nu-1$ and the extension $\Phi$ of $\phi$
exists, then choosing $\chi$ such that $\Phi=\phi$ on the support
of $\dbar\chi$ we have that
$$
R_{n-\nu}\w \dbar\chi\w h\Phi\w d\zeta=
d(R_{n-\nu}\w \chi\w h\Phi\w d\zeta)
$$
and since $R_{n-\nu}\w \chi\w h\Phi\w d\zeta$ has compact support
\eqref{valfrid} must hold.

If  $\chi$ is as  in the theorem, then $(1-\chi)\phi$ is a smooth
extension of $\phi$ across $\overline{\B_\epsilon}$, and to find the
$\dbar$-closed extension we have to solve $\dbar u=f$
with compact support, where $f=\dbar\chi\w\phi$. 
To this end, 
let $\tilde\chi$ be a cutoff function that is $1$ in a \nbh of 
a closed ball that contains the
support of $f$ and let  $g$ be the weight from
Example~\ref{alba}  with this choice of $\tilde\chi$ but with 
$z$ and $\zeta$ interchanged. It does not have
compact support with respect to $\zeta$, but  since $f$ has compact
support itself we still  have the Koppelman formula \eqref{koppelman3}.
Clearly 
$$
u(z)=\int (HR\w g\w B)_n\w  f
$$
has support in a \nbh of the support of $f$, 
and it follows from Koppelman's formula that it is indeed a solution if 
the associated integral $\Pr f$ vanishes. However, since now $s$
is holomorphic in $\zeta$,  for degree reasons we have that
\begin{multline*}
\Pr f(z)=\int (HR\w g)_n\w f = \pm \dbar\tilde{\chi}(z)\w \int HR_{n-q-1}\w
s\w (\dbar s)^q\w\dbar\chi\w\phi \\
=\pm \dbar\tilde{\chi}(z)\w \int HR_{n-q-1}\w
\frac{\bar z\cdot d \zeta\w(d\bar z\cdot d\zeta)^{q}}
{(2\pi i(|z|^2-\bar z\cdot\zeta))^{q+1}}\w \dbar\chi\w\phi.
\end{multline*}
If  $q<\nu-1$, then this integral vanishes since then $R_{n-q-1}=0$.
If $q=\nu-1$, then $\Pr\phi$  vanishes if \eqref{valfrid} holds, keeping in mind
that $H$ is holomorphic in the ball. 
Since $f=0$ in a \nbh of $0$ in $Z$ we have that
$u$ is smooth, and 
$\Phi=(1-\chi)\phi+u$ is the desired $\dbar$-closed extension.
\end{proof}

In particular we have proved a simple case of
Theorem~\ref{portensats} and we obtain 
the general case along the same lines.

\begin{proof}[Proof of Theorem~\ref{portensats}]
Since $X$ can be exhausted by holomorphically convex subsets each of which
can be embedded in some affine space, we can assume from the beginning that
$X\subset\Omega\subset\C^n$, where $\Omega$ is pseudoconvex.
Let $\omega\subset\subset\Omega$ be a holomorphically convex open set in $\Omega$ that
contains  $K$. Let $\chi$ be a cutoff function in $\omega$ that is
$1$ in a \nbh of $K$.
Choose a cutoff function $\tilde\chi$ that is $1$ in a \nbh of the 
holomorphically convex hull of the support
of $f$  and let $g$ be the weight from
Example~\ref{alba}  with this choice of $\tilde\chi$ but with 
$z$ and $\zeta$ interchanged.  As in the previous proof 
we get a solution with support in $\omega$, 
provided that the  corresponding projection term $\Pr f$ %%vanishes. 
vanishes. If $\nu>1$,  then  $\Pr f$ vanishes automatically and if
$\nu=1$, then $\Pr f=0$ if
\begin{equation}\label{alla}
\int R_{n-1}\w d\zeta \w\dbar\chi\w\phi h=
\pm \int_{X}\dbar\chi\w \phi h \w (\gamma_{n-1}\lrcorner d\zeta)=0
\end{equation}
for all $h\in\Ok(\omega\cap X)$,  and by approximation
it is enough to assume that \eqref{alla}  holds for $h\in\Ok(X)$,
i.e., that \eqref{moment} holds.

Since  $X_{sing}$ is not contained in $K$, our solution
$u$ is, outside of $K$, only defined on $X_{reg}$.
Therefore
$\Phi=(1-\chi)\phi+u$ is holomorphic in $X_{reg}$, in a
\nbh of $K$, and outside $\omega$. Since
$X_{reg}^\ell\setminus K$ is connected, $\Phi=\phi$ 
there.  
(Even without  the connectedness assumptions it follows that
$\Phi$ is in $\Ok(X)$, since it has at most polynomial growth
at $Z_{sing}$, hence  is meromorphic and its pole set is contained
in $\omega\cap X$.)
The necessity of the moment condition follows
as in the previous proof.
\end{proof}

\begin{ex}\label{hartogsex}
Let $X\subset \C^2$ be an irreducible curve with one transverse self intersection at $0\in \C^2$.
Close to $0$, $X$ has two irreducible components, $A_1$, $A_2$, each isomorphic to a disc in $\C$.
Let $K\subset A_1$ be a closed annulus surrounding the intersection point $A_1\cap A_2$. Then 
$X\setminus K$ is connected but $X_{reg} \setminus K$ is not. Denote 
the ``bounded component'' of $A_1\setminus K$ by $U_1$ and put $U_2=X\setminus (K\cup U_1)$.
Let $\tilde{\phi}\in \mathcal{O}(X)$ satisfy $\tilde{\phi}(0)=0$ and define $\phi$ to be $0$
on $U_1$ and equal to $\tilde{\phi}$ on $U_2$. Then $\phi\in \mathcal{O}(X\setminus K)$ and a straight
forward verification shows that $\phi$ satisfies the compatibility condition \eqref{moment}; cf.\ also 
\eqref{oscar}. But clearly, $\phi$ cannot be extended to a strongly holomorphic function on $X$.     
\end{ex}

We now consider the case when $X_{sing}$ has positive dimension more closely. 
Locally we have an analogue of Proposition~\ref{badanka2}.
For convenience  we first  consider  the technical part
concerning  solutions  with compact support.

\begin{prop}\label{badanka3}
Let $Z$ be an analytic set defined in a neighborhood of $\bar{\B}\subset\C^n$,
let $x\in Z_{sing}$, and let $a$ be a holomorphic tuple such that
$Z_{sing}=\{a=0\}$ in a neighborhood of $x$ and let $d'=\dim Z_{sing}$.
Assume that $f$ is a smooth $\dbar$-closed $(0,q)$-form in a
neighborhood of $x$ with $f=0$ close to $Z_{sing}$ and with f supported in
$\{|a|<t\}$ for some small $t$.
\begin{itemize}

\item[(i)] If $1\leq q \leq \nu-d'-1$, then in a neighborhood $U$ of $x$ one can
find a  smooth $(0,q-1)$-form, $u$, with support in $\{|a|<t\}$
and $\dbar u =f$ in $U\cap Z_{reg}$.

\item[(ii)] If $q=\nu-d'$, then one can find such a solution if and only if

\begin{equation}\label{eqmoment}
\int R_{n-\nu}\wedge h\wedge f=
\pm\int_{Z}f\w h\w (\gamma_{n-\nu}\lrcorner d\zeta)=0
\end{equation}
for all smooth $\dbar$-closed $(0,d')$-forms, $h$, such that
$\mbox{supp}(h)\cap \{|a|\leq t\}$ is compact.
\end{itemize}
\end{prop}

\begin{proof}
Let $\chi_a$ be a cutoff function in $\B$, which in a neighborhood of $x$
satisfies that $\chi_a=1$ in a neighborhood of the support of $f$ and
$\chi_a=0$ in a neighborhood of $\{|a|\geq t\}$. Let also
$H^a$ be a holomorphic $(1,0)$-form, as in the previous section, and define

\begin{equation*}
g^a=\chi_a(z)-\dbar\chi_a(z)\wedge \frac{\sigma_a}{\nabla_{\eta}\sigma_a},\,\,\,
\sigma_a=\frac{\overline{a(z)}\cdot H^a}{|a(z)|^2-a(\zeta)\cdot \overline{a(z)}}.
\end{equation*}
Then $g^a$ is a smooth weight for $\zeta$ in the support of $f$.
Close to $x$ we can choose coordinates
$(z',z'')=(z_1',\ldots,z'_{d'},z_1'',\ldots,z_{p+r}'')$ centered at $x$ so
that
$Z_{sing}\subset \{|z''|\leq |z'|\}$. Since $f$ is supported close to
$Z_{sing}$ we can choose a function $\chi=\chi(\zeta')$, which is $1$
close to
$x$ and $f\chi$ has compact support. Let now
$g=\chi-\dbar\chi\wedge \sigma/\nabla_{\eta}\sigma$ be the weight from
Example 2 but built from $z'$ and $\zeta'$. Our Koppelman formula now gives
that
%%%
\begin{equation*}
u=\K f=\int (HR\wedge g^a \wedge g\wedge B)_n\wedge f
\end{equation*}
has the desired properties provided that the obstruction term
%%%%
\begin{equation*}
\Pr f=\int (HR\wedge g^a\wedge g)_n\wedge f
\end{equation*}
vanishes. Since $g$ is built from $\zeta'$, $g$ has at most degree $d'$ in
$d\bar{\zeta}$. Moreover,
$HR$ has at most degree $n-\nu$ in $d\bar{\zeta}$ and
$g^a$ has no degree in $d\bar{\zeta}$. Thus, if $q<\nu-d'$, then
$(HR\wedge g^a\wedge g)_n\wedge f$ cannot have degree $n$ in $d\bar{\zeta}$
and so $\Pr f=0$ in that case. This proves (i).

To show (ii), note that if $q=\nu-d'$, then 
\begin{equation*}
\Pr f=\chi_a(z)\int HR_{n-\nu}\wedge g_{d'}\wedge f.
\end{equation*}
Now, $H$ depends holomorphically on $\zeta$ and $g_{d'}$ is
$\dbar$-closed since it is the top degree term of a weight. Also, $g$ has
compact support in the $\zeta'$-direction, so $\mbox{supp}(g)\cap\{|a|\leq
t\}$
is compact and thus $\Pr f=0$ if \eqref{eqmoment} is fulfulled.
On the other hand, it is clear that the existence of a solution
with support in $\{|a|<t\}$ implies \eqref{eqmoment}.
\end{proof}

%%Proposition~\ref{badanka3} of course implies
%%results analogous to Proposition~\ref{badanka2}.

\begin{proof}[Proof of Theorem~\ref{main3}] %%
We first assume that 
$\Omega=\B$ and $Z\subset\Omega$ has the single singular point $0$.
If $q=0<\nu-1$ (or $q=0=\nu-1$ and \eqref{valfrid} holds), then
it is clear from Proposition~\ref{badanka2} that $\phi$ is strongly
holomorphic.

Fix $r<1$ and let  $K_\ell=Z\cap(\overline\B_r\setminus\B_{1/\ell})$.
If now $q<\nu-1$
it follows from Proposition~\ref{badanka2} that there is a
$\dbar$-closed form  $\Phi_\ell$ in a \nbh in $Z$ of  $\overline\B_r\cap Z$
that coincides with $\phi$ in a \nbh of $K_\ell$, and
by Theorem~\ref{main} we therefore have  a smooth solution $u_\ell'$
to $\dbar u_\ell'=\phi$ in a \nbh of $K_\ell$.
Now $u_{\ell+1}'- u_\ell'$ is a $\dbar$-closed $(0,q-1)$-form
in a \nbh of $K_\ell$ and thus there is a global smooth $\dbar$-closed
form $w_\ell$ that coincides with $u_{\ell+1}'- u_\ell'$
in a \nbh of $K_\ell$. 
If we let $u_k=u_k'-(w_1+\cdots +w_{k-1})$ then $u=\lim u_k$ exists
and solves $\dbar u=\phi$ in $Z\cap \B_r\setminus\{0\}$.
%% Then $u_{\ell+1}'-u_\ell'$ is a $\dbar$-closed $(0,0)$-form
%%in a \nbh of $K_\ell$ and since now $0<\nu-1$ we know that 
%%it has a strongly holomorphic extension $w_\ell$. In fact, our  integral formula
%%provides a holomorphic extension to a whole ball  $\overline\B_r$.
%%In particular $w_\ell$ is holomorphic in a \nbh of
%%
%%
%%If $q>1$ we get $\dbar$-closed forms $w_\ell$ in $\Z\cap\overline{\B_r}$
%%such that $u_{\ell+1}'-u'_\ell=w_\ell$
%%
%%instead  that $u_{\ell+1}'-u_\ell'=\dbar v_\ell$
%%in a \nbh of $K_\ell$ and we get the desired (semi-)global solution
%%in a similar  way.

Notice that if the desired solution exists, then \eqref{valfrid}
must be fulfilled.

%%%%\begin{proof}[Proof of Theorem~\ref{main3}, $\dim Z_{sing}>0$]
%%

\smallskip
Assume now that $X$ is an analytic space with
arbitrary singular set. 
Arguing as in the proof of the case $\dim X_{sing}=0$ above, we can conclude
from Proposition~\ref{badanka3}: {\it Given a point $x$ there is a \nbh
$U$ such that if $\phi$ is a $\dbar$-closed smooth 
$(0,q)$-form  in $U\cap X_{reg}$, $0\le q<\nu-d'-1$,  then 
$\phi$ is strongly holomorphic if $q=0$ and exact in
$X_{reg}\cap U'$,  for a  possibly
slightly smaller \nbh $U'$ of $x$,  if $q\ge 1$.}

\smallskip
We define the analytic sheaves
$\F_k$ on $X$ by 
$\F_k(U)=\E_{0,k}(U\cap X_{reg})$ for open sets $U\subset X$. Then
$\F_k$ are fine sheaves and 
\begin{equation}\label{karv}
0\to\Ok_X \to\F_0\stackrel{\dbar}{\longrightarrow}\F_1
\stackrel{\dbar}{\longrightarrow}\F_2\stackrel{\dbar}{\longrightarrow} \cdots
\end{equation}
is exact for $k<\nu-d'-1$.
It follows that 
$$
H^k(X,\Ok_X)=\frac{\Ker_{\dbar} \F_k(X)}
{\dbar \F_{k-1}(X)}
$$
for $k<\nu-d'-1$. Hence Theorem~\ref{main3} follows. 
\end{proof}

\begin{proof}[Proof of Theorem~\ref{main3x}]
%%Assume now that $Z_{sing}$ is discrete and $q=d=\dim Z=n-p$. 
We first assume that $X\subset\Omega\subset\C^n$
has an isolated  singularity at $0$. After a linear change of
coordinates in $\C^n$, and shrinking $\Omega$,
we may assume that the $d$-tuple
$a(z)=(z_1,\ldots, z_d)$ vanishes only at $0$ on $X$.
%%In a suitable  neighborhood $U$ of  $0$ we may 
%%assume, after a linear change of coordinates, that 
%%$Z\cap U\subset\{|z''|<|z'|\}$, where
%%$z=(z',z'')=(z_1',\ldots,z_{n-p}',z_1'',\ldots, z_p'')$.
%%In particular the tuple $a(z)=z'$ has $0$ as common
%%zero set.  
%%
Let $U_\ell=\{|a|<2^{-\ell}\}\cap \Omega$. 
We claim that if $f$ is a smooth $(0,d)$-form 
in $U_{\ell}\setminus \{0\}$,
with support in $U_{\ell}$, then there is a smooth form $v_\ell$ such that
$f-\dbar v_\ell$ has support in $U_{\ell+1}$ and 
$v_\ell$ together with its derivatives up to order $\ell$ are bounded by
$2^{-\ell}$ outside $U_\ell$.

From the beginning we assume that $\phi$ has support in $U_1$.
Taking the claim for granted we choose inductively 
$f$ as $\phi-\dbar v_1-\ldots - \dbar v_{\ell-1}$, and we then obtain a solution
$
v=v_1+v_2+\ldots
$
in $U\setminus\{0\}$ to $\dbar v=\phi$.

To see the claim we use the weight \eqref{viktig} %%from Section~\ref{??}
but with $z$ and $\zeta$ interchanged, i.e.,
$$
g^\mu=(\sigma(z)\cdot a(\zeta)+\dbar\sigma(z)\cdot H^1)^\mu,
$$
where $\sigma=\bar a/|a|^2$.
After a small modification we may assume that $f$ 
vanishes identically in a \nbh of $0$. Then since
$f$ has support in $U_\ell$,
$$
\K f(z)=\int (HR\w g^\mu\w B)_n\w f
$$
together with a finite number of derivatives will be small outside
$U_\ell$ if $\mu$ is chosen large enough. As before it is smooth
since $f=0$ close to $Z_{sing}$.
Moreover it is a solution, because
$$
\Pr f(z)=\int (HR\w g^\mu)_n\w f
$$
will vanish for degree reasons since $\dbar\sigma_1\w\ldots\w\dbar\sigma_{d}=0$.

\smallskip
Finally assume that $X$ is a general Stein space. 
%% Since we cannot solve
%%$\dbar$ for $q<d$ we cannot use the same technique to globalize. However,
%%Since $Z_{sing}$ is discrete there is a simple direct way:
Since we can solve $\dbar u=\phi$
in a \nbh of each singular point, we can find a global  $u$ such that
$f=\phi-\dbar u$ is smooth and vanishes in a  \nbh of $X_{sing}$. By
Theorem~\ref{mainglobal} we can solve $\dbar v=f$ on $X_{reg}$
and thus $\dbar (v+u)=\phi$ in $X_{reg}$. 
\end{proof}

\section{Meromorphic and strongly holomorphic functions}\label{merosec}

A meromorphic function $\phi$ on $Z\subset\Omega$ can be represented
by a meromorphic $\Phi$ in the ambient space
that is generically holomorphic on
$Z_{reg}$. Let $R$ be the residue current associated with $Z$. 
We show in \cite{A13} that $R\phi$ is well-defined
for any meromorphic $\phi$.   In fact, it can be defined  as
the  analytic continuation to $\lambda=0$ of the current
$|h|^{2\lambda}\Phi R,$
if $\Phi$ is a representative of $\phi$
in the ambient space and $h$ is a holomorphic function in $\Omega$ such that $h\Phi$ is
holomorphic and generically non-vanishing on $Z$.
One also has a well-defined current
$$
R\w\dbar\phi=-\nabla_f(R \phi)=\dbar|h|^{2\lambda}\w R\phi|_{\lambda=0}
$$
with support on the pole  set $P_\phi$ of $\phi$.

In \cite{A13} we proved the following result that
generalizes a previous result by  Tsikh in the case of
a complete intersection, see \cite{Ts} and \cite{HP}.

\begin{thm}
If $\phi$ is meromorphic on $Z$, then
$\phi$ is strongly holomorphic  if and only if
$
 R\w\dbar\phi=0.
$
\end{thm}

%%Notice that if $Z$ is CM then the condition is precisely that
%%$\dbar\phi R_p)=0$, and recall that in case complete ablba

%%\smallskip

%%Moreover, 
%%$$
%%\dbar[\phi]\w R=\dbar|h|^{2\lambda}\Phi\w R
%%$$
%%is well-defined and in fact
%%$$
%%-\dbar[\phi]\w R=\nabla_f ([\phi]R).
%%$$

By our  Koppelman formula
we can give a proof that provides  an explicit analytic extension
of $\phi$ to $\Omega$.

\begin{proof}
Assume that $\phi$ is meromorphic on $Z$ and let $\Phi$ be
a representative.  For 
$\Re\lambda>>0$ we have from Theorem~\ref{main}, 
$$
|h(z)|^{2\lambda}\Phi(z)=
\int  |h|^{2\lambda}HR \Phi\w g+ 
\int \dbar|h|^{2\lambda}\w  HR\Phi\w g\w  B.
$$
For $z \in Z_{reg}\setminus \{h=0\}$ we can take $\lambda=0$ and 
we get (after choosing various $h$)
the formula 
%%$B(\zeta-z)$ is smooth in a \nbh of
%%where $\phi$ is not defined or $R$ is singular, so we get
$$
\phi(z)=\int  HR\phi\w g+
\int H(R\w\dbar\phi)\w g
\w B, \quad z\in Z_{reg}\setminus P_\phi.
$$
If $R\w\dbar\phi=0$ it follows that $\phi(z)$ generically is equal
to the first term on the right hand side which is a strongly holomorphic function. 
\end{proof}

We conclude by formulating a conjecture.
If $\phi$ is weakly holomorphic then $P_\phi\subset Z_{sing}$
so $R\w\dbar\phi$ has support
on $Z_{sing}$. Since $R\w\dbar\phi$ 
is a $\HM$-current it follows for degree reasons that
it  must vanish if 
\begin{equation}\label{2v}
\codim Z_{sing}\ge 2+p,\quad \codim Z_k\ge 2+k,\ k>p,
\end{equation}
see \cite{A13}.
This means that all weakly holomorphic functions are indeed
strongly holomorphic if \eqref{2v} is fulfilled.
One can check that \eqref{2v} is  equivalent to the
conditions $R1$ and $S2$ in Serre's criterion, 
see, e.g., \cite{Eis1}. Therefore \eqref{2v} is
indeed equivalent to that all  (germs of) weakly holomorphic functions
are holomorphic, i.e., $Z$ is a normal variety.

Suppose that $\phi$ is a smooth $\dbar$-closed $(0,q)$-form in
$Z_{reg}$ and assume that $\phi$  admits some reasonable extension
across $Z_{sing}$ so that $R\w\dbar\phi$ is a hypermeromorphic
current. Arguing as in \cite{A13}  it  follows that
$R\w\dbar\phi$ must vanish if
\begin{equation}\label{qv}
\codim Z_{sing}\ge 2+q+p,\quad \codim Z_k\ge 2+q+k,\ k>p,
\end{equation}
which is (equivalent to) the conditions $R_{q-1}$ and $S_q$.
The  Koppelman formula will then produce a smooth solution to
$\dbar\psi=\phi$ on $Z_{reg}$.
One could therefore conjecture that the Dolbeault cohomology groups
$H^{0,\ell}(Z_{reg})$ vanish for $\ell\le q$  if (and only if?)    \eqref{qv} holds.

If we consider $Z$ as an intrinsic analytic space, then 
in the notation in Remark~\ref{obsz} the condition \eqref{qv}
means that $\codim Z^r\ge 2+q+r$ for $r\ge 0$.

%%\section{Mer spekulativt}

%%Inspired by previous section we define currents on $Z$;
%%let $a$ be an intrinsic  current on $Z_{reg}$ ,
%%%%i.e., it acts on $\D(Z_{reg})$. We say that it is a current
%%on $Z$,  $a\in\Cu(Z)$,   if 
%%$$
%%\lim_{\epsilon}\chi(|h|/\epsilon)u.\xi
%%$$
%%exists for all $\xi\in\D(Z)$. Here $h$ is a tuple such that
%%$\{h=0\}=Z_{sing}$ or at least that $\{h=0\}\cap Z=Z_{sing}$.
%%Check that def independent of choice of $h$.

%%\smallskip

%%If $a\in \Cu(Z)$ then $\dbar A\in\Cu(Z)$  ?? i vilken mening da???

%%\bigskip
%%Hur som helst ska man nu tillampa grundformeln pa 
%%$\chi(|h|/\epsilon)\phi$, sag tex $phi$ smooth om $Z_{reg}$ for
%%simplicity. Then we have

%%$$
%%\phi(z)=\dbar \int g^z\w H \chi_\epsilon\phi R\w B +
%%\int g^z H \dbar\chi_\epsilon \phi R \w B
%%$$
%%so if lats integral tends to $0$ we get at least a solution
%%in $Z_{reg}$.

%%\bigskip 

%%Notice that $\dbar a=0$ means that  ?????

%%???????????????????

\def\listing#1#2#3{{\sc #1}:\ {\it #2},\ #3.}

\end{document}